\documentclass[12pt,a4paper,twoside]{amsart}
\usepackage{charter}
\usepackage{mathrsfs}
\usepackage{enumitem}
\usepackage{amsmath, amsfonts, amssymb, amsthm, amscd}
\usepackage[all]{xy}
\usepackage{fancyhdr}
\usepackage{hyperref}

\usepackage[mathscr]{eucal}

\begin{document}
\newtheorem{theorem}{Theorem}[section]
\newtheorem{thmdef}{Theorem-definition}[section]
\newtheorem{lemma}[theorem]{Lemma}
\newtheorem{lemmadef}[theorem]{Lemma-definition}
\newtheorem{proposition}[theorem]{Proposition}
\newtheorem{corollary}[theorem]{Corollary}
\newtheorem{conjecture}[theorem]{Conjecture}
\newtheorem{observation}[theorem]{Observation}

\theoremstyle{definition}
\newtheorem{definition}[theorem]{Definition}
\newtheorem{example}[theorem]{Example}
\newtheorem{examples}[theorem]{Examples}
\newtheorem{none}[theorem]{}
\theoremstyle{remark} 
\newtheorem{remark}[theorem]{Remark}
\newtheorem{caution}[theorem]{Caution}
\newtheorem{remarks}[theorem]{Remarks}
\newtheorem{question}[theorem]{Question}
\newtheorem{problem}{Problem}
\newtheorem{exercise}{Exercise}

\renewcommand{\bar}{\overline}
\def\C{\mathbb{C}}
\def\T{\mathcal{T}}
\def\X{\mathcal{X}}
\def\U{\mathcal{U}}
\def\P{\Phi}
\def\M{\mathcal{M}}
\def\Z{\mathcal{Z}_{m}}
\def\ZZ{\mathcal{Z}_{m}^{H}}
\def\d{\partial}
\def\cP{\mathscr P}
\def\bB{\mathbb B}
\def\TT{\mathcal{T}_{m}^{H}}

\def\i{\sqrt{-1}}

\title{Moduli spaces as ball quotients I, local theory}
\author{Kefeng Liu}
\address{School of Mathematics, Capital Normal University, Beijing, China;
Department of Mathematics, University of California at Los Angeles,
Los Angeles, CA 90095-1555, USA}
\email{liu@math.ucla.edu, kefeng@cms.zju.edu.cn}
\author{Yang Shen}
\address{Center of Mathematical Sciences, Zhejiang University, Hangzhou, Zhejiang 310027, China}
\email{syliuguang2007@163.com}

\begin{abstract} 
We give a Hodge theoretic criterion to characterize locally the period domains of certain projective manifolds 
to be complex balls.
\end{abstract}

\maketitle

\tableofcontents






\setcounter{section}{-1}
\section{Introduction}
In \cite{ACT02}, Allcock, Carlson and Toledo study the moduli space of cubic surfaces.
By studying the compactifications of the moduli spaces and the corresponding points in the period domain, they proved the global Torelli theorem for cubic surfaces.
Moreover, they showed that the moduli space of stable cubic surfaces can be realized as a ball quotient.
Several years later they also proved similar results for cubic threefolds in \cite{ACT11}.

Therefore, it is natural to find certain conditions for the polarized manifold 
under which the corresponding moduli space can be realized as a ball quotient.
The main tool is the period map, under which the problem becomes the following.

\begin{problem}\label{ball problem}
(1). Find certain conditions for polarized manifolds $X$ under which the period domain for $X$ has a submanifold $B$, which is isomorphic to the complex ball, such that the image of the period map can be embedded into $B$; 

(2). Under (1), one can construct the so-called refined period map with values in $B$. Then one may ask: when the refined period map is injective?  
\end{problem}

Problem \ref{ball problem} - (1) usually concerns the local theory of period maps, and Problem \ref{ball problem} - (2) concerns the global theory of period maps.

In this paper, as the title suggests, we mainly focus on Problem \ref{ball problem} - (1). We give a intrinsic characterization of the required conditions of the polarized manifolds for Problem \ref{ball problem} - (1), which is defined in Definition \ref{ball type}.

In the subsequent paper \cite{LS18}, we will prove global Torelli theorem for polarized manifolds satisfying the conditions of Definition \ref{ball type}, which answers Problem \ref{ball problem} - (2).

We call the polarized manifold satisfying the required conditions for Problem \ref{ball problem}, as given in Definition \ref{ball type}, a polarized manifold of ball type. Roughly speaking, for a polarized manifold $(X,L)$ of ball type, the first nonzero subspace $H_{\mathrm{pr}}^{k,n-k}(X)$ in the Hodge decomposition 
$$H_{\mathrm{pr}}^{n}(X,\C)=H_{\mathrm{pr}}^{n,0}(X)\oplus H_{\mathrm{pr}}^{n-1,1}(X)\oplus \cdots \oplus H_{\mathrm{pr}}^{0,n}(X)$$
satisfying the following properties:

(i)  There exists $\Omega \in H_{\mathrm{pr}}^{k,n-k}(X)$ such that $\{\theta_{i} \lrcorner \Omega\}_{1\le i\le N}$, is linearly independent in $H_{\mathrm{pr}}^{k-1,n-k+1}(X)$, where $\{\theta_{i}\}_{1\le i\le N}$ is the basis of $H^{1}(X,\Theta_{X})$.

(ii) $\theta_{i}\lrcorner\theta_{j}\lrcorner (H_{\mathrm{pr}}^{k,n-k}(X))=0$, for any $1\le i,j \le N$.

The first property means that the infinitesimal Torelli theorem holds for such polarized manifolds. As is proved in this paper, the second property annihilates the higher order terms in the local expansion of the canonical section of the Hodge bundle $\mathcal H^{k,n-k}$. Precisely, as the main result of this paper, we will prove the following theorem.

\begin{theorem}\label{intr linear expansion}
Let $X$ be the polarized manifold of ball type.
Let $f:\, \mathcal X\to S$ be a family of polarized manifolds of ball type over the complex
manifold $S$, containing the base point $p\in S$ such that $X\simeq f^{-1}(p)$.
Under the canonical coordinate $\{U;z^c\}$ (to be defined in Definition \ref{canonical coordinate definition}) around the base point $p$ of $S$, we have the following expansion of the holomorphic section $\Omega(z^c)$ for the polarized manifolds of ball type as
\begin{align}\label{linear expansion0}
\Omega(z^c)=\Omega_0+\sum_{1\le i\le N}[\theta^c_i \lrcorner\Omega_0]z_i^c.
\end{align}
\end{theorem}

From the Hodge-Riemann bilinear relations, we derive that 
$$-1+\sum_{1\le i\le N}|z_i^c|^{2}<0,$$ which tells us that the image of the refined period map lies in a unit ball $\mathbb{B}^N$ in a complex Euclidean space $\mathbb{C}^N$. 

Moreover, we can define the refined period map
$$\cP:\, S \to \Gamma \backslash \mathbb{P}(H),$$
from the base manifold $S$ of the analytic family $f:\, \mathcal X \to S$ of polarized manifolds of ball type, such that any $q\in S$ is mapped to the complex line
$$L_{q}=\mathrm{Span}_{\C}\{\Omega_{q}\} \subset H_{\mathrm{pr}}^{n}(X_{q},\C)$$
 spanned by $\Omega_{q}$ in $H_{\mathrm{pr}}^{k,n-k}(X_{q})$ in Definition \ref{ball type}.
Here $$H:=H_{\mathrm{pr}}^{n}(X_{p}, \C), \, p\in S \text{ is the base point},$$ 
such that $H$ is isomorphic to $H_{\mathrm{pr}}^{n}(X_{q}, \C)$, and $\Gamma$ is the corresponding monodromy group.
Then Theorem \ref{intr linear expansion} is equivalent to the following theorem.

\begin{theorem}
The refined period domain $\mathbb{P}(H)$ is isomorphic to the unit ball $\mathbb{B}^N$ in a complex Euclidean space $\mathbb{C}^N$. Moreover,
the refined period map
$$\cP:\, S \to \Gamma \backslash \mathbb{B}^N$$
is locally isomorphic.
\end{theorem}


\section{Discussions on the basic assumptions}\label{basic assump}
Let $(X,L)$ be a polarized manifold, i.e. $X$ is a projective manifold and $L$ is an ample line bundle on $X$. For simplicity, we sometimes simply call $X$ a polarized manifold.

Now let us introduce the definition of polarized manifolds of ball type, which is used to characterize the corresponding period domain admitting a complex ball as subset. 

\begin{definition}\label{ball type}
The polarized manifold $(X,L)$ is said to be of ball type if the following conditions are satisfied: 
\begin{itemize}
\item[(i)] There exists some integer $k$ with $0\le k\le n$ such that
$$H_{\mathrm{pr}}^{l,n-l}(X)=0$$
for $k\le l\le n$.

\item[(ii)] There is an element $\Omega$ in $H_{\mathrm{pr}}^{k,n-k}(X)$ such that, if $\{\theta_{i}\}_{1\le i\le N}$ is the basis of $H^{1}(X,\Theta_{X})$, then 
\begin{equation}
\{\theta_{i} \lrcorner \Omega\}_{1\le i\le N} \textrm{ is linearly independent in } H_{\mathrm{pr}}^{k-1,n-k+1}(X), \tag{$*_1$} 
\end{equation}
and 
\begin{equation}
\theta_{i}\lrcorner\theta_{j}\lrcorner (H_{\mathrm{pr}}^{k,n-k}(X))=0, \text{ for any }1\le i,j \le N. \tag{$*_2$}
\end{equation}
\end{itemize}
\end{definition}

%
%

The motivation of the above conditions is to construct the decomposition of the cohomology groups $H_{\mathrm{pr}}^{n}(X,\C)$, which is invariant under deformation. Here the notion of deformation invariant cohomology groups is defined as follows. 

First recall that a deformation of the polarized manifold $X$ over the polydisc $\Delta$ is given by an analytic family $\pi:\, \mathcal{X}\to \Delta$ of polarized manifolds such that $X_{t}:\,=\pi^{-1}(t)$ is a polarized manifold for any $t \in \Delta$, and the central fiber $X_{0}:\,=\pi^{-1}(0)$ is identified with $X$.

By classical deformation theory, we know that the analytic family $\pi:\, \mathcal{X}\to \Delta$ admits a transversely holomorphic trivialization
$$F_{\sigma}=(\sigma,\pi):\, \mathcal{X} \to X_{0}\times \Delta,$$
with the $C^{\infty}$ projection map $\sigma:\, \mathcal{X} \to X_{0}$ whose fibers are complex holomorphic polydisks meeting $X_{0}$ transversely. For any $t \in \Delta$, $F_{\sigma}$ induces a diffeomorphism 
$$\sigma_{t}:\, X_{t} \to X_{0}.$$
See the paper of Clemens \cite{Clemens05} for the details of the transversely holomorphic trivialization.

\begin{definition}\label{deformation invariant}
Let $X$ be a polarized manifold. 

(1) The subgroup $$H_{0}^{n}(X,\C) \subset H^{n}(X,\C)$$
is called stable under deformation, provided that for any deformation $\pi:\, \mathcal{X}\to \Delta$ of $X$ with $X_{0}=X$, there exist sections $\alpha_{1}(t), \cdots , \alpha_{r}(t)$ of $(\sigma_{t}^{-1})^{*}(H^{n}(X_{t},\C))=H^{n}(X,\C)$, which are holomorphic in $t\in \Delta$, such that 
$$(\sigma_{t}^{-1})^{*}(H_{0}^{n}(X_{t},\C))= \C\{\alpha_{1}(t), \cdots , \alpha_{r}(t)\}$$
where $\sigma_{t}:\, X_{t} \to X_{0}$ is the diffeomorphism defined as above;

(2) The decomposition 
$$H^{n}(X,\C)=H_{0}^{n}(X,\C)\oplus H_{c}^{n}(X,\C)$$
is called stable under deformation, 
provided that there exist sections $$\alpha_{1}(t), \cdots , \alpha_{r}(t), \alpha_{r+1}(t), \cdots , \alpha_{h^{n}}(t)$$ of $(\sigma_{t}^{-1})^{*}(H^{n}(X_{t},\C))=H^{n}(X,\C)$, which are simultaneously holomorphic in $t\in \Delta$, such that 
$$(\sigma_{t}^{-1})^{*}(H_{0}^{n}(X_{t},\C))= \C\{\alpha_{1}(t), \cdots , \alpha_{r}(t)\}$$
and
$$(\sigma_{t}^{-1})^{*}(H_{c}^{n}(X_{t},\C))= \C\{\alpha_{r+1}(t), \cdots , \alpha_{h^{n}}(t)\},$$
where $h^{n}=\dim_{\C} H^{n}(X,\C)$.
\end{definition}

\begin{definition}\label{deformation invariant1}
Let $X$ be a polarized manifold. 
The subgroup $$H_{0}^{n}(X,\C) \subset H^{n}(X,\C)$$
is called invariant under deformation, or deformation invariant, provided that for any deformation $\pi:\, \mathcal{X}\to \Delta$ of $X$ with $X_{0}=X$, there exists a subgroup $$H_{0}^{n}(X_{t},\C) \subset H^{n}(X_{t},\C)$$ for any $t \in \Delta$ , which is independent of the choice the deformation $\pi:\, \mathcal{X}\to \Delta$, such that 
$$(\sigma_{t}^{-1})^{*}(H_{0}^{n}(X_{t},\C))=H_{0}^{n}(X,\C).$$
\end{definition}

From Definition \ref{deformation invariant}, we know that, if there exists a line bundle $\mathcal L$ on the total space $\mathcal X$ of the family $\pi:\, \mathcal X \to \Delta$ such that $\mathcal L|_{X_{t}}$ is ample on $X_{t}$ for any $t\in \Delta$, then $H_{\mathrm{pr}}^{n}(X,\C)$ is a deformation invariant subgroup of itself. 

\begin{remark}
By the definition of period maps, we know that the subgroup
$$F^{k}H_{\mathrm{pr}}^{n}(X,\C)\subset H_{\mathrm{pr}}^{n}(X,\C)$$
is stable under deformation for $0\le k\le n$.
While, due to Griffiths transversality, the Hodge decomposition 
$$H_{\mathrm{pr}}^{n}(X,\C) =\bigoplus_{p+q=n} H_{\mathrm{pr}}^{p,q}(X)$$
is NOT stable under deformation.
\end{remark}


\begin{remark}
Under conditions (i) and (ii) of polarized manifolds of ball type in Definition \ref{ball type}, the subgroup
\begin{equation}\label{di 1}
\C\{\Omega, \theta_{1}\lrcorner \Omega, \cdots, \theta_{N}\lrcorner \Omega\}\subset F^{k-1}(X)=H_{\mathrm{pr}}^{k,n-k}(X) \oplus H_{\mathrm{pr}}^{k-1,n-k+1}(X)
\end{equation}
is stable under deformation as a subgroup of $H_{\mathrm{pr}}^{n}(X,\C)$.
See Corollary \ref{omega stable} below for the proof.

Let $H_{\Omega}:\,= \C\{\Omega, \theta_{1}\lrcorner \Omega, \cdots, \theta_{N}\lrcorner \Omega\}$. Then as an application of the main results in our paper, c.f. Theorem \ref{linear expansion} and Lemma \ref{linear expansion'}, we can conclude directly that 
$$ H_{\Omega} \oplus \bar{H}_{\Omega} \subset H_{\mathrm{pr}}^{n}(X,\C)$$
is deformation invariant.
\end{remark}


Deformation invariant subgroups of $H_{\mathrm{pr}}^{n}(X,\C)$ can determine the geometry of the moduli space.
We come up with this notion via understanding the results in \cite{ACT02} and \cite{ACT11} of Allcock, Carlson, and Toledo.

We introduce the notion of eigenperiod domain to overview the main idea of \cite{ACT02} and \cite{ACT11}. See \cite{DK}, or \cite{LS15}, for the detailed discussion of eigenperiod domain and the corresponding eigenperiod map.

Let $\pi:\, \mathcal{X} \to \Delta$ be the analytic family. 
For convenience, we assume that the family $\pi$ is universal.
Suppose that there exits a line bundle $\mathcal{L}$ on $\mathcal{X}$ such that
the restriction $L_{t}=\mathcal{L}|_{X_{t}}$ on $X_{t}=\pi^{-1}(t)$, for any $t\in \Delta$, is an ample line bundle on the polarized manifold $X_{t}$.

Let $G$ be a finite abelian group acting holomorphically on $\mathcal{X}$, preserving the line bundle $\mathcal{L}$ on $\mathcal{X}$. 
We assume that $g(X_{t})=X_{t}$ for any $t\in \Delta$ and $g\in G$.
Then the diffeomorphism 
$$\sigma_{t}:\, X_{t} \to X_{0}.$$
gives the identification $$(H_{\mathrm{pr}}^{n}(X_{t},\mathbb{Z})/\text{Tor},Q)\stackrel{\sim}{\longrightarrow}(H_{\mathrm{pr}}^{n}(X_{0},\mathbb{Z})/\text{Tor},Q)$$ for any $t\in \Delta$, where $Q$ is the Poincar\'e paring.
Since $G$ preserves the line bundle $\mathcal{L}$ on $\U$, we have a induced action of $G$ on $H_{\mathrm{pr}}^{n}(X_{0},\mathbb{Z})/\text{Tor}$ preserving $Q$, i.e. we have a representation $$\rho:\, G\to \text{Aut}(H_{\mathrm{pr}}^{n}(X_{0},\mathbb{Z})/\text{Tor},Q).$$

Let $\chi\in \text{Hom}(G,\C^{*})$ be a character of $G$. Define the eigenspace
$$H_{\chi}^n(X,\C)=\{v\in H_{\mathrm{pr}}^n(X,\C):\, \rho(g)(v)=\chi(g)v, \forall g\in G\}.$$
Since the action of $G$ is defined on $\mathcal{X}$ and the family $\pi$ is universal, we see that eigenspace $$H_{\chi}^n(X,\C) \subset H_{\mathrm{pr}}^n(X,\C)$$ is deformation invariant, and so is the decomposition 
$$H_{\mathrm{pr}}^n(X,\C)=\bigoplus_{\chi\in \text{Hom}(G,\C^{*})} H_{\chi}^n(X,\C).$$

The eigenperiod domain is then defined by the set of all the Hodge decompositions 
$$H_{\chi}^n(X,\C) =\bigoplus_{p+q=n} H_{\chi}^{p,q}(X,\C),$$
of the eigenspace $H_{\chi}^n(X,\C)$, where $$H_{\chi}^{p,q}(X,\C) :\, =H_{\mathrm{pr}}^{p,q}(X,\C) \cap H_{\chi}^n(X,\C).$$ See \cite{LS15} for the precise definition.

We will compute the cases of cubic surfaces, cubic threefolds and Deligne-Mostow 
theory for corresponding eigenperiod domain, in Section \ref{Examples}.


\section{Period domains and Lie groups}\label{Lie}

Let $H_{\mathbb{Z}}$ be a fixed lattice and $H=H_{\mathbb{Z}}\otimes_{\mathbb{Z}}
\C$ the  complexification. Let $n$ be a positive integer, and $Q$ a
bilinear form on $H_{\mathbb{Z}}$ which is symmetric if $n$ is even
and skew-symmetric if $n$ is odd. Let $h^{i,n-i}$, $0\le i\le n$, be
integers such that $\sum_{i=0}^{n}h^{i,n-i}=\dim_{\C}H$. The period
domain $D$ for the polarized Hodge structures of type
$$\{H_{\mathbb{Z}}, Q, h^{i,n-i}\}$$ is the set of all the collections
of the subspaces $H^{i,n-i}$, $0\le i\le n$, of $H$ such that
$$H=\bigoplus_{0\le i\le n}H^{i,n-i}, \,H^{i,n-i}=\overline{H^{n-i,i}}, \, \dim_{\C} H^{i,n-i}= h^{i,n-i} \text{ for } 0\le i\le n,$$
and on which $Q$ satisfies the  Hodge-Riemann bilinear relations,
\begin{eqnarray}
Q\left ( H^{i,n-i}, H^{j,n-j}\right )=0\text{ unless }i+j=n;\label{HR1}\\
\left (\sqrt{-1}\right )^{2k-n}Q\left ( v,\bar v\right )>0\text{ for
}v\in H^{k,n-k}\setminus\{0\}. \label{HR2}
\end{eqnarray}

Alternatively, in terms of Hodge filtrations, the period domain $D$
is the set of all the  collections of the filtrations
$$H=F^{0}\supset F^{1}\supset \cdots \supset F^{n},$$ such that
\begin{align}
& \dim_{\mathbb{C}} F^i=f^i,  \label{periodcondition} \\
& H=F^{i}\oplus \overline{F^{n-i+1}},\text{ for } 0\le i\le n,\nonumber
\end{align}
where $f^{i}= h^{n,0}+\cdots +h^{i,n-i}$, and on which $Q$ satisfies the  Hodge-Riemann bilinear relations in the form of Hodge filtrations
\begin{align}
& Q\left ( F^i,F^{n-i+1}\right )=0;\label{HR1'}\\
& Q\left ( Cv,\bar v\right )>0\text{ if }v\ne 0,\label{HR2'}
\end{align}
where $C$ is the Weil operator given by $$Cv=\left (\sqrt{-1}\right )^{2k-n}v$$ for $v\in F^{k}\cap \overline{F^{n-k}}$.



Let $(X,L)$ be a polarized manifold with $\dim_\C X=n$, which means that
 $X$ is a projective manifold and $L$ is an ample line bundle on
 $X$. 
The $n$-th primitive cohomology groups
$H^n(X,\C)$ of $X$ is defined by
$$H_{\mathrm{pr}}^n(X,\C)=\ker\{\cdot \wedge c_{1}(L):\, H^n(X,\C)\to H^{n+2}(X,\C)\}.$$

 Let $\Phi:\, S\to \Gamma\backslash D$ be the period map
from geometry.  More precisely we have an algebraic family $$f:\, \X\to S$$ of polarized algebraic manifolds over a quasi-projective manifold $S$, such that for
any $q\in S$, the point $\Phi(q)$, modulo certain action of the monodromy group $\Gamma$, represents the Hodge structure of the $n$-th primitive cohomology group
$H_{\mathrm{pr}}^{n}(X_{q},\C)$ of the fiber $X_{q}=f^{-1}(q)$. Here $H\simeq H_{\mathrm{pr}}^{n}(X_{q},\C)$ for any $q\in S$.

Recall that the monodromy group $\Gamma$, or global monodromy group, is the image of the representation of $\pi_1(S)$
in  $\text{Aut}(H_{\mathbb{Z}},Q)$. 
Here $\text{Aut}(H_{\mathbb{Z}},Q)$ denotes the group of automorphisms of $H_{\mathbb{Z}}$ preserving $Q$.

By taking a finite index torsion-free subgroup of $\Gamma$,  we can assume that $\Gamma$ is
torsion-free,  therefore $\Gamma\backslash D$ is smooth. This way we can just proceed on a finite cover of $S$ without loss of generality.  We refer
the reader to the proof of Lemma IV-A, pages 705 -- 706 in
\cite{Sommese} for such standard construction.

Since period map is locally liftable, we can lift the period map to $\P : \T \to D$ by taking the universal cover $\T$ of $S$ such that the diagram
\begin{equation}\label{periodlifting}
\xymatrix{
\T \ar[r]^-{\P} \ar[d]^-{\pi} & D\ar[d]^-{\pi}\\
S \ar[r]^-{\Phi} & \Gamma\backslash D
}
\end{equation}
is commutative.

Now we fix a point $p$ in $\T$ and its image $o=\P(p)$ as the reference points or base points, and denote the Hodge decomposition corresponding to the point $o=\P(p)$ as $$H_{\mathrm{pr}}^n(X_p, {\mathbb{C}})=H^{n, 0}_p\oplus H^{n-1, 1}_p\oplus\cdots \oplus H^{1, n-1}_p\oplus H^{0, n}_p, $$
 where $H_{p}^{i,n-i}=H_{\mathrm{pr}}^{i,n-i}(X_{p})$ for $0\le i \le n$, and the Hodge filtration by
 $$H_{\mathrm{pr}}^n(X_p, {\mathbb{C}})=F_{p}^{0}\supset F_{p}^{1}\supset \cdots \supset F_{p}^{n}$$
with $F_{p}^{i}=H_{p}^{n,0}\oplus \cdots \oplus H_{p}^{i,n-i}$ for $0\le i \le n$.

Let us introduce the notion of adapted basis for the given Hodge decomposition or Hodge filtration.
We call a basis $$\xi=\left\{ \xi_0, \cdots, \xi_{f^{n}-1},\xi_{f^{n}}, \cdots ,\xi_{f^{n-1}-1}, \cdots, \xi_{f^{k+1}}, \cdots, \xi_{f^k-1}, \cdots, \xi_{f^{1}},\cdots , \xi_{f^{0}-1} \right\}$$ of $H^n(X_p, \mathbb{C})$ an adapted basis
 for the given Hodge decomposition if it satisfies
 $$H^{k, n-k}_p=\text{Span}_{\mathbb{C}}\left\{\xi_{f^{k+1}}, \cdots, \xi_{f^k-1}\right\}.$$
We call a basis
\begin{align*}
\zeta=\left\{ \zeta_0, \cdots, \zeta_{f^{n}-1},\zeta_{f^{n}}, \cdots ,\zeta_{f^{n-1}-1}, \cdots, \zeta_{f^{k+1}}, \cdots, \zeta_{f^k-1}, \cdots,\zeta_{f^{1}},\cdots , \zeta_{f^{0}-1} \right\}
\end{align*}
of $H^n(X_p, {\mathbb{C}})$ an adapted basis for the given filtration
if it satisfies $$F^{k}_p=\text{Span}_{\mathbb{C}}\{\zeta_0, \cdots, \zeta_{f^k-1}\}.$$
For convenience, we set $f^{n+1}=0$ and $m=f^0$.

\begin{definition}\label{blocks}
(1) The blocks of an $m\times m$ matrix $\Psi=(\Psi_{ij})_{0\le i,j\le m-1}$ are set as follows. For each
$0\leq \alpha, \beta\leq n$, the $(\alpha, \beta)$-th block
$\Psi^{(\alpha, \beta)}$ is defined by
\begin{align}\label{block}
\Psi^{(\alpha, \beta)}=\left(\Psi_{ij}\right)_{f^{-\alpha+n+1}\leq i \leq f^{-\alpha+n}-1, \ f^{-\beta+n+1}\leq j\leq f^{-\beta+n}-1}.
\end{align}
In particular, $\Psi =(\Psi^{(\alpha,\beta)})_{0\le \alpha,\beta \le n}$ is called a {block lower triangular matrix} if
$\Psi^{(\alpha,\beta)}=0$ whenever $\alpha<\beta$.

(2) The blocks of the adapted basis $$\xi=\left\{ \xi_0, \cdots, \xi_{f^{n}-1}, \cdots , \xi_{f^{k+1}}, \cdots, \xi_{f^k-1}, \cdots, \xi_{f^{1}},\cdots , \xi_{f^{0}-1} \right\}$$ are defined by
$$\xi_{(\alpha)}=\{\xi_{f^{-\alpha +n+1}}, \cdots, \xi_{f^{-\alpha +n}-1}\},$$
for $0\le \alpha \le n$. Then 
$$\xi=\{\xi_{(0)},\cdots, \xi_{(n)}\}.$$
\end{definition}

Let $H_{\mathbb{F}}=H_{\mathrm{pr}}^n(X, \mathbb{F})$, where $\mathbb{F}$ can be chosen as $\mathbb{Z}$, $\mathbb{R}$, $\mathbb{C}$. Then $H=H_{\mathbb{C}}$ under this notation. We define the complex Lie group
\begin{align*}
G_{\mathbb{C}}=\{ g\in GL(H_{\mathbb{C}})|~ Q(gu, gv)=Q(u, v) \text{ for all } u, v\in H_{\mathbb{C}}\},
\end{align*}
and the real one
\begin{align*}
G_{\mathbb{R}}=\{ g\in GL(H_{\mathbb{R}})|~ Q(gu, gv)=Q(u, v) \text{ for all } u, v\in H_{\mathbb{R}}\}.
\end{align*}
We also have $$G_{\mathbb Z}=\text{Aut}(H_{\mathbb{Z}},Q) =\{ g\in GL(H_{\mathbb{Z}})|~ Q(gu, gv)=Q(u, v) \text{ for all } u, v\in H_{\mathbb{Z}}\}.$$

 Griffiths in \cite{Griffiths1} showed that $G_{\mathbb{C}}$ acts on
$\check{D}$ transitively, so does $G_{\mathbb{R}}$ on $D$. The
stabilizer of $G_{\mathbb{C}}$ on $\check{D}$ at the base point $o$
is
$$B=\{g\in G_\C| gF_p^k=F_p^k,\ 0\le k\le n\},$$
and the one of $G_{\mathbb{R}}$ on $D$ is $V=B\cap G_\mathbb{R}$.
Thus we can realize $\check{D}$, $D$ as
$$\check{D}=G_\C/B,\text{ and }D=G_\mathbb{R}/V$$
so that $\check{D}$ is an algebraic manifold and $D\subseteq
\check{D}$ is an open complex submanifold.

The Lie algebra $\mathfrak{g}$ of the complex Lie group $G_{\mathbb{C}}$ is
\begin{align*}
\mathfrak{g}&=\{X\in \text{End}(H_\mathbb{C})|~ Q(Xu, v)+Q(u, Xv)=0, \text{ for all } u, v\in H_\mathbb{C}\},
\end{align*}
and the real subalgebra
$$\mathfrak{g}_0=\{X\in \mathfrak{g}|~ XH_{\mathbb{R}}\subseteq H_\mathbb{R}\}$$
is the Lie algebra of $G_\mathbb{R}$. Note that $\mathfrak{g}$ is a
simple complex Lie algebra and contains $\mathfrak{g}_0$ as a real
form,  i.e. $\mathfrak{g}=\mathfrak{g}_0\oplus \i\mathfrak{g}_0$.

On the linear space $\text{Hom}(H_\C,H_\C)$ we can give a Hodge structure of weight zero by
\begin{align*}
\mathfrak{g}=\bigoplus_{k\in \mathbb{Z}} \mathfrak{g}^{k,
-k}\quad\text{with}\quad\mathfrak{g}^{k, -k}= \{X\in
\mathfrak{g}|~XH^{r, n-r}_p\subseteq H^{r+k, n-r-k}_p,\ \forall r
\}.
\end{align*}

By the definition of $B$, the Lie algebra $\mathfrak{b}$ of $B$ has
the  form $\mathfrak{b}=\bigoplus_{k\geq 0} \mathfrak{g}^{k, -k}$.
Then the Lie algebra $\mathfrak{v}_0$ of $V$ is
$$\mathfrak{v}_0=\mathfrak{g}_0\cap \mathfrak{b}=\mathfrak{g}_0\cap \mathfrak{b} \cap\bar{\mathfrak{b}}=\mathfrak{g}_0\cap \mathfrak{g}^{0, 0}.$$
With the above isomorphisms, the holomorphic tangent space of $\check{D}$ at the base point is naturally isomorphic to $\mathfrak{g}/\mathfrak{b}$.

Let us consider the nilpotent Lie subalgebra
$\mathfrak{n}_+:=\oplus_{k\geq 1}\mathfrak{g}^{-k,k}$.  Then one
gets the isomorphism $\mathfrak{g}/\mathfrak{b}\cong
\mathfrak{n}_+$. We denote the corresponding unipotent Lie group to be
$$N_+=\exp(\mathfrak{n}_+).$$

As $\text{Ad}(g)(\mathfrak{g}^{k, -k})$ is in  $\bigoplus_{i\geq
k}\mathfrak{g}^{i, -i} \text{ for each } g\in B,$ the subspace
$\mathfrak{b}\oplus \mathfrak{g}^{-1, 1}/\mathfrak{b}\subseteq
\mathfrak{g}/\mathfrak{b}$ defines an Ad$(B)$-invariant subspace. By
left translation via $G_{\mathbb{C}}$,
$\mathfrak{b}\oplus\mathfrak{g}^{-1,1}/\mathfrak{b}$ gives rise to a
$G_{\mathbb{C}}$-invariant holomorphic subbundle of the holomorphic
tangent bundle. It will be denoted by $\mathrm{T}^{1,0}_{h}\check{D}$,
and will be referred to as the horizontal tangent subbundle. One can
check that this construction does not depend on the choice of the
base point.

The horizontal tangent subbundle, restricted to $D$, determines a subbundle $\mathrm{T}_{h}^{1, 0}D$ of the holomorphic tangent bundle $\mathrm{T}^{1, 0}D$ of $D$.
The $G_{\mathbb{C}}$-invariance of $\mathrm{T}^{1, 0}_{h}\check{D}$ implies the $G_{\mathbb{R}}$-invariance of $\mathrm{T}^{1, 0}_{h}D$. Note that the horizontal
tangent subbundle $\mathrm{T}_{h}^{1, 0}D$ can also be constructed as the associated bundle of the principle bundle $V\to G_\mathbb{R} \to D$ with the adjoint
representation of $V$ on the space $\mathfrak{b}\oplus\mathfrak{g}^{-1,1}/\mathfrak{b}$.

Let $\mathscr{F}^{k}$, $0\le k \le n$ be the  Hodge bundles on $D$
with fibers $\mathscr{F}^{k}|_{s}=F_{s}^{k}$ for any $s\in D$. 
Let $\mathscr{H}^{p,q}=\mathscr{F}^{p}/\mathscr{F}^{p+1}$, $p+q=n$, be the quotient bundles such that $\mathscr{H}^{p,q}|_{s}={H}^{p,q}_{s}$.
As another
interpretation of the horizontal bundle in terms of the Hodge
bundles $\mathscr{F}^{k}\to \check{D}$, $0\le k \le n$, one has
\begin{align}\label{horizontal}
\mathrm{T}^{1, 0}_{h}\check{D}\simeq \mathrm{T}^{1, 0}\check{D}\cap \bigoplus_{k=1}^{n}\text{Hom}(\mathscr{F}^{k}/\mathscr{F}^{k+1}, \mathscr{F}^{k-1}/\mathscr{F}^{k}).
\end{align}

\begin{remark}
We remark that the elements in $N_+$ can be realized as nonsingular block lower triangular matrices with identity blocks in the diagonal; the elements in $B$ can be realized as nonsingular block upper triangular matrices.
If $c, c'\in N_+$ such that $cB=c'B$ in $\check{D}$, then $$c'^{-1}c\in N_+\cap B=\{I \},$$ i.e. $c=c'$. This means that the matrix representation in $N_+$ of the unipotent orbit $N_+(o)$ is unique. Therefore with the fixed base point $o\in \check{D}$, we can identify $N_+$ with its unipotent orbit $N_+(o)$ in $\check{D}$ by identifying an element $c\in N_+$ with $[c]=cB$ in $\check{D}$. Therefore our notation $N_+\subseteq\check{D}$ is meaningful. In particular, when the base point $o$ is in $D$, we have $N_+\cap D\subseteq D$.

\end{remark}


\begin{lemma}\label{transversal}Let $p\in\mathcal{T}$ be the base point with $\P(p)=\{F^n_p\subseteq F^{n-1}_p\subseteq \cdots\subseteq F^0_p\}.$ Let $q\in \mathcal{T}$ be any point with $\P(q)=\{F^n_q\subseteq F^{n-1}_q\subseteq \cdots\subseteq F^0_q\}$, then $\P(q)\in N_+$ if and only if $F^{k}_q$ is isomorphic to $F^k_p$ for all $0\leq k\leq n$.
\end{lemma}
\begin{proof}For any $q\in \T$, we choose an arbitrary adapted basis $\{\zeta_{(0)}, \cdots, \zeta_{(n)}\}$ for the given Hodge filtration $\{F^{n}_q\subseteq F^{n-1}_q\subseteq\cdots\subseteq F^0_q\}$. We fix $\{\eta_{(0)}, \cdots, \eta_{(n)}\}$ as the adapted basis for the Hodge filtration $\{F^{n}_p \subseteq F^{n-1}_p\subseteq \cdots\subseteq F^0_p\}$ at the base point $p$. Let $(A^{(\alpha, \beta)}(q))_{0\le \alpha, \beta \le n}$ be the transition matrix between the basis $\{\eta_{(0)}, \cdots, \eta_{(n)}\}$ and $\{\zeta_{(0)}, \cdots, \zeta_{(n)}\}$ for the same vector space $H_\C$, where $A^{(\alpha, \beta)}(q)$ are the corresponding blocks. Then $$\P(q)\in N_+=N_+B/B\subseteq \check{D}$$ if and only if its matrix representation $(A^{(\alpha, \beta)}(q))_{0\le \alpha, \beta \le n}$ can be decomposed as $L(q)\cdot U(q)$, where $L(q)$ is a nonsingular block lower triangular matrix with identities in the diagonal blocks, and $U(q)$ is a nonsingular block upper triangular matrix.

By basic linear algebra, we know that $(A^{(\alpha, \beta)}(q))$ has such decomposition if and only if $\det(A^{(\alpha, \beta)}(q))_{0\leq \alpha, \beta \leq k}\neq 0$ for any $0\leq k\leq n$. In particular, we know that $(A^{(\alpha, \beta)}(q))_{0\leq \alpha, \beta\leq k}$ is the transition map between the bases of $F^k_p$ and $F^k_q$. Therefore, $\det((A^{(\alpha, \beta)}(q))_{0\leq \alpha, \beta\leq k})\neq 0$ if and only if $F^k_q$ is isomorphic to $F^k_p$.
\end{proof}

\begin{proposition}\label{codimension}
The subset $N_{+}$ is an open complex submanifold in $\check{D}$, and $\check{D}\setminus N_{+}$ is an analytic subvariety of $\check{D}$ with $\text{codim}_{\mathbb{C}}(\check{D}\setminus N_{+})\geq 1$.
\end{proposition}
\begin{proof}
From Lemma \ref{transversal}, one can see that $\check{D}\setminus N_+\subseteq \check{D}$ is defined as an analytic subvariety by the equations
\begin{align*}
\{q\in \check{D} : \det ((A^{(\alpha, \beta)}(q))_{0\leq \alpha, \beta\leq k})=0\text{ for some } 0\leq k\leq n\}.
\end{align*}
Therefore $N_+$ is dense in $\check{D}$, and that $\check{D}\setminus N_+$ is an analytic subvariety,
which is closed in $\check{D}$ and with $\text{codim}_{\mathbb{C}}(\check{D}\backslash N_+)\geq 1$.
\end{proof}

\section{The expansions of local sections of the Hodge bundles}\label{Expansions of Local sections}
\subsection{Basic lemmas}\label{bl}
In this section, we consider the period map in general. Let $\Phi:\, S\to \Gamma \backslash D$ be the period map, which is equivalent to say that $\Phi$ is locally liftable and holomorphic, and $\Phi$ satisfies Griffiths transversality.
Let $U\subset S$ be any small enough open subset, and then
$$\Phi:\, U\to D$$
is the locally lifted period map.

In Definition \ref{blocks}, we have introduced the blocks of the adapted basis $\eta$ as
$$\eta=\{\eta_{(0)},\eta_{(1)},\cdots,\eta_{(n)}\},$$
where $\eta_{(\alpha)}=\{\eta_{f^{-\alpha +n+1}}, \cdots, \eta_{f^{-\alpha +n}-1}\}$ is the basis of $H^{n-\alpha,\alpha}(X_p)$ for $0\le \alpha \le n$.

By Proposition \ref{codimension}, we know that $N_+\cap D$ is an open dense subset of $D$ with $D\setminus (N_+\cap D)$ an analytic subset of $D$. 
Then we can shrink the small neighborhood $U$ of $S$ such that
$$\Phi(U)\subset N_+\cap D.$$
Then the lifted period map becomes
$$\Phi:\, U \to N_+\cap D.$$

For any $q\in U$, we can choose the matrix representation of the image $\Phi(q)$ in $N_+$ by $$\P(q)=(\Phi_{ij}(q))_{0\le i,j\le m-1}\in N_+\cap D.$$
Using the matrix representation, we can construct the holomorphic sections of the Hodge bundles $\mathscr{H}^{n-\alpha,\alpha}$ over $U$, for $0\le \alpha\le n$, as
\begin{eqnarray}\label{lm section}
\Omega_{(\alpha)}(q)&=&\sum_{\beta=0}^{n}\eta_{(\beta)}\cdot\Phi^{(\beta,\alpha)}(q) \nonumber \\
&=&\eta_{(\alpha)}+\sum_{\beta\ge \alpha+1}\eta_{(\beta)}\cdot\Phi^{(\beta,\alpha)}(q),
\end{eqnarray}
where $\eta_{(\beta)}\cdot\Phi^{(\beta,\alpha)}(q)$, as a product of two matrices, is a basis of $H^{n-\beta,\beta}(X_p)$, and $\Omega_{(\alpha)}(q)$ is a basis of $H^{n-\alpha,\alpha}(X_{q})$.

Let $\{V;z=(z_1,\cdots,z_N)\}$ be any small holomorphic coordinate around the base point $p$ with $z_\mu(p)=0$, $1\le \mu\le N$, such that $V\subset U$ and 
$$V=\{z\in \mathbb{C}^N:\, |z_{\mu}|<\epsilon,\ 1\le \mu\le N\}.$$
For $0\le \alpha,\beta \le n$, we define the matrix value functions as
$$\Phi^{(\alpha,\beta)}(z)=\Phi^{(\alpha,\beta)}(q), \text{ for } q\in U \text{ and } z(q)=z.$$

In the following, the derivatives of the blocks,
$$\frac{\partial \P^{(\alpha,\beta)}}{\partial z_\mu}(z), \text{ for } 0\le \alpha,\beta \le n,\ 1\le \mu\le N,$$
will be denoted to be the blocks of derivatives of its entries,
$$\frac{\partial \P^{(\alpha,\beta)}}{\partial z_\mu}(z)=\left(\frac{\partial \Phi_{ij}}{\partial z_\mu}(z)\right)_{f^{-\alpha+n+1}\leq i \leq f^{-\alpha+n}-1, \ f^{-\beta+n+1}\leq j\leq f^{-\beta+n}-1}.$$

\begin{lemma}\label{lm derivative lemma}
Let the notations be as above. Then we have that 
\begin{align}\label{lm derivative}
\frac{\partial \P^{(\alpha,\beta)}}{\partial z_\mu}(z)=\P^{(\alpha,\beta+1)}(z)\cdot \frac{\partial \P^{(\beta+1,\beta)}}{\partial z_\mu}(z)
\end{align}
for any $0\le \beta <\alpha\le n$ and $1\le \mu\le N$.
\end{lemma}
\begin{proof}
The main idea of the proof is to understand the Griffiths transversality in terms of the matrix representations of the image of the period map in $N_+$.

Let us consider the holomorphic sections $$\Omega_{(\beta)}(z)=\Omega_{(\beta)}(q), \text{ for } q\in U \text{ and } z(q)=z$$ of the Hodge bundles $\mathscr{H}^{n-\beta,\beta}$ in \eqref{lm section}. By Griffiths transversality,  especially the computations in Page 813 of \cite{Griffiths2}, we have that 
$$\frac{\partial \Omega_{(\beta)}}{\partial z_\mu}(z)$$
lies in ${H}^{n-\beta,\beta}(X_z)\oplus {H}^{n-\beta-1,\beta+1}(X_z)$, where $X_z=X_q$ for $z(q)=z$.

By the constructions of the holomorphic sections in \eqref{lm section}, we know that 
$$\Omega_{(\beta)}(z) \text{ and }\Omega_{(\beta+1)}(z)$$
are bases of $H^{n-\beta,\beta}(X_z)$ and ${H}^{n-\beta-1,\beta+1}(X_z)$ respectively. Therefore, there exist matrices $A_\mu(z)$ and $B_{\mu}(z)$ such that 
\begin{eqnarray}\label{1}
\frac{\partial \Omega_{(\beta)}}{\partial z_\mu}(z)&=&\Omega_{(\beta)}(z)A_{\mu}(z)+\Omega_{(\beta+1)}(z) \cdot B_\mu(z) \nonumber \\
&\equiv&\eta_{(\beta)}A_{\mu}(z) \text{ mod } \bigoplus_{\gamma\ge 1}{H}^{n-\beta-\gamma,\beta+\gamma}(X_p),
\end{eqnarray}
where the inclusion \eqref{1} is deduced by comparing the types in the holomorphic sections in \eqref{lm section}.

Again, by looking at the holomorphic sections in \eqref{lm section}, we have that 
\begin{eqnarray}\label{2}
\frac{\partial \Omega_{(\beta)}}{\partial z_\mu}(z)&=&\sum_{\alpha\ge \beta+1}\eta_{(\alpha)}\cdot \frac{\partial \Phi^{(\alpha,\beta)}}{\partial z_\mu}(z) \nonumber \\
&=& \eta_{(\beta+1)}\cdot\frac{\partial\Phi^{(\beta+1,\beta)}}{\partial z_\mu}(z)+\sum_{\alpha\ge \beta+2}\eta_{(\alpha)}\cdot \frac{\partial \Phi^{(\alpha,\beta)}}{\partial z_\mu}(z)\nonumber\\
&\equiv&0 \text{ mod } \bigoplus_{\gamma\ge 1}{H}^{n-\beta-\gamma,\beta+\gamma}(X_p).
\end{eqnarray}

By comparing equations \eqref{1} and \eqref{2}, we see that $A_{\mu}(z)=0$, and hence 
\begin{eqnarray*}
\frac{\partial \Omega_{(\beta)}}{\partial z_\mu}(z)&=&\Omega_{(\beta+1)}(z) \cdot B_\mu(z) \\
&=&\left(\eta_{(\beta+1)}+\sum_{\alpha\ge \beta+2}\eta_{(\alpha)}\cdot\Phi^{(\alpha,\beta+1)}(z)\right)\cdot B_\mu(z)\\
&=& \eta_{(\beta+1)}\cdot B_\mu(z)+\sum_{\alpha\ge \beta+2}\eta_{(\alpha)}\cdot\Phi^{(\alpha,\beta+1)}(z)\cdot B_\mu(z).
\end{eqnarray*}
%

By comparing types again, we have that
\begin{align*}
B_\mu(z)= \frac{\partial\Phi^{(\beta+1,\beta)}}{\partial z_\mu}(z)
\end{align*}
and
\begin{align}\label{lm derivative'}
\frac{\partial\Phi^{(\alpha,\beta)}}{\partial z_\mu}(z)=\Phi^{(\alpha,\beta+1)}(z)\cdot \frac{\partial \Phi^{(\beta+1,\beta)}}{\partial z_\mu}(z),
\end{align}
for $\alpha \ge \beta+2$. Since $\Phi^{(\beta+1,\beta+1)}(z)$ is identity matrix, equation \eqref{lm derivative'} is satisfied trivially for $\alpha=\beta+1$. Therefore, we have proved equation \eqref{lm derivative}.
\end{proof}

\begin{corollary}\label{omega stable}
Under conditions (i) and (ii) of polarized manifolds of ball type in Definition \ref{ball type}, the subgroup
\begin{equation}\label{di 1}
\C\{\Omega, \theta_{1}\lrcorner \Omega, \cdots, \theta_{N}\lrcorner \Omega\}\subset F^{k-1}(X)=H_{\mathrm{pr}}^{k,n-k}(X) \oplus H_{\mathrm{pr}}^{k-1,n-k+1}(X)
\end{equation}
is stable under deformation as a subgroup of $H_{\mathrm{pr}}^{n}(X,\C)$.
\end{corollary}
\begin{proof}
Let the notations be as in Definition \ref{deformation invariant}.

By the construction in \eqref{lm section}, for $\Omega \in H_{\mathrm{pr}}^{k,n-k}(X)$, we can construct a holomorphic section $\Omega(t) \in H_{\mathrm{pr}}^{k,n-k}(X_{t})$ such that $\Omega(0)=\Omega$ and 
$$\frac{\partial }{\partial t_{i}}\Omega(t) \in  F^{k-1}(X_{t})/F^{k}(X_{t}), \, 1\le i\le N.$$ 
Since Condition ($*_{1}$) of (ii) in Definition \ref{ball type} is an open condition, we have that $\Omega(t)$ still satisfy Condition ($*_{1}$) for $t$ small enough.

By Griffiths transversality and by comparing types, we have that $$\frac{\partial }{\partial t_{i}}\Omega(t)=\theta_{i}(t)\lrcorner \Omega(t)$$
where $\theta_{i}(t)$ is the image of $\frac{\partial }{\partial t_{i}}|_{t}$ under the Kodaira-Spencer map at $t$. Therefore $\frac{\partial }{\partial t_{i}}\Omega(t)$ is a section of $\C\{\theta_{i}(t)\lrcorner \Omega(t)\}_{i=1}^{N}\subset F^{k-1}(X_{t})/F^{k}(X_{t})$. Finally, we have constructed holomorphic sections $\Omega(t), \frac{\partial }{\partial t_{1}}\Omega(t), \cdots \frac{\partial }{\partial t_{N}}\Omega(t)$ of $F^{k-1}(X_{t})$ such that they spans $\C$-linearly the subgroup in \eqref{di 1} for $X_{t}$.
\end{proof}

\begin{proposition}\label{lm expansion proposition}
Let the notations be as above. Then we have the following expansions of the blocks,
\begin{align}\label{lm expansion}
 \P^{(\alpha, \beta)}(z)=O(|z|^{\alpha-\beta}), \text{ for } 0\le \beta \le \alpha\le n 
\end{align}
\end{proposition}
\begin{proof}
Equation \eqref{lm expansion} is trivial for $\alpha=\beta$. 

At the base point $p$, the matrix representation in $N_+$ of $\P(p)$ is the identity matrix, and hence $\P^{(\alpha, \beta)}(0)=0$ for $0\le \beta < \alpha\le n$. That implies that the expansions of the blocks $\P^{(\alpha, \beta)}(z)$ have no constant terms, and equation \eqref{lm expansion} holds for $\alpha=\beta+1$.

Now we prove the proposition inductively. We may assume that equation \eqref{lm expansion} holds for $\alpha=\beta+s$, i.e. $\P^{(\alpha, \beta)}(z)=O(|z|^{s})$, where $1\le s < n$. 

Let us consider the case for $\alpha=\beta+s+1$. By the equation \eqref{lm derivative} in Lemma \ref{lm derivative lemma}, we have
$$\frac{\partial \P^{(\alpha, \beta)}}{\partial z_\mu}(z)=\P^{(\alpha, \beta+1)}(z)\cdot \frac{\partial \P^{(\beta+1, \beta)}}{\partial z_\mu}(z),$$
where $\P^{(\alpha, \beta+1)}(z)=O(|z|^{s})$ by induction, and, from the discussion at the beginning of the proof, $\frac{\partial \P^{(\beta+1, \beta)}}{\partial z_\mu}(z)$ is bounded when $z$ is near the origin. Therefore we conclude that 
$$\frac{\partial \P^{(\alpha, \beta)}}{\partial z_\mu}(z)=O(|z|^{s}),$$
which implies that 
$$\P^{(\alpha, \beta)}(z)=O(|z|^{s+1}).$$

At last we finish the proof by induction.
\end{proof}


\subsection{Calabi--Yau type manifolds}\label{GCY}
In this section, we first define the notion of Calabi--Yau type manifolds.

\begin{definition}\label{defn of gcyt}
Let $X$ be a projective manifold with $\dim_{\mathbb{C}}X=n$. We call $X$ a Calabi--Yau type manifold if it satisfies the following conditions:
\begin{enumerate}
\item There exists some $k$, satisfying $[n/2]< k\leq n$, such that
\begin{align*}
H_{\mathrm{pr}}^{l,n-l}(X)=0 \text{ for } l>k,\text{ and }\dim_{\mathbb{C}}H_{\mathrm{pr}}^{k,n-k}(X)=1;
\end{align*}
\item For any generator $[\Omega]\in H_{\mathrm{pr}}^{k,n-k}(X)$, the contraction map
\begin{align*}
\lrcorner:\, H^{1}(X,\Theta_X)\to H_{\mathrm{pr}}^{k-1,n-k+1}(X), \quad\quad
[\phi]\mapsto [\phi\lrcorner\Omega]
\end{align*}
is an isomorphism.
\label{condition2}
\end{enumerate}
\end{definition} 


The motivation of introducing the notion of Calabi--Yau type manifolds is to simplify the proof of the expansions of local sections of the Hodge bundles for polarized manifolds of ball type. 

Note that, with condition (i) and ($*_{1}$) of (ii) in Definition \ref{ball type}, a polarized manifold $X$ such that $\mathrm{dim}_{\C}H^{k,n-k}(X)=1$ is a Calabi--Yau type manifold.

Let $f:\,\mathcal{X}\to S$ be a family of Calabi--Yau type manifolds over the complex manifold $S$. 
We are considering the local properties of the period map in this section.
Therefore, we take a small enough neighborhood $U$ of $S$, such that the period map 
$$\Phi:\, U \to D$$
is well-defined.
The period map $\Phi$ then gives the Hodge bundles $$\mathscr{H}^{l,n-l}:\,= \mathscr{F}^{l}/\mathscr{F}^{l+1}$$ over $U$, for $n-k\le l\le k$.

Fix $p\in U\subset S$ as the base point, we can choose any holomorphic coordinate $\{U;z\}$ around the base point $p$ such that $z(p)=0$. Then the basis $\{\theta_{i}\}_{1\le i\le m}$ of $H^{1}(X,\Theta_{X})$ can be given by the images of the holomorphic vectors
$$\frac{\partial}{\partial z_i}\bigg|_0\in \mathrm{T}_p U,\ 1\le i\le N,$$
under the Kodaira-Spencer map.
By equation (2) of Definition \ref{defn of gcyt}, $\{\theta_{i}\lrcorner \Omega_0\}_{1\le i\le N}$ is linearly independent in $H_{\mathrm{pr}}^{k-1,n-k+1}(X)$.

Let us study the local section of the Hodge bundle $\mathscr{H}^{k,n-k}$ for Calabi--Yau type manifolds.

We fix the base point $p$ in $U$. Due to Definition \ref{defn of gcyt}, the Hodge structure of Calabi--Yau type manifold $X$ is of the form
$$H=H^{k,n-k}\oplus H^{k-1,n-k+1}\oplus \cdots \oplus H^{n-k,k},$$
where $H^{p,n-p}\simeq H_{\mathrm{pr}}^{p,n-p}(X)$, for $n-k\le p\le k$.
Then the adapted basis $\eta$ as given in Definition \ref{blocks} is of the form
$$\eta=(0,\cdots,0,\eta_{(n-k)},\eta_{(n-k+1)},\cdots, \eta_{(k)},0,\cdots, 0).$$
The corresponding matrices in $N_{+}\cap D$ is of the form 
\begin{align*}
\Phi=\left[\begin{array}[c]{ccccccc}
\ddots&\vdots&\vdots&  & \vdots &   \vdots\\ 
\cdots & I &0&\cdots &0 &\cdots   \\ 
\cdots & \Phi^{(n-k+1,n-k)} & I          & \cdots     &0      & \cdots  \\ 
\cdots & \vdots & \vdots & \ddots     &\vdots       & \cdots  \\
\cdots     & \Phi^{(k,n-k)}     &  \Phi^{(k,n-k+1)}   &  \cdots & I &  \cdots \\
\vdots & \vdots  & \vdots  & \cdots     &\vdots       & \ddots  
\end{array}\right].
\end{align*}
Since $\dim_{\mathbb{C}}H^{k,n-k}(X)=1$, the block $\Phi^{(n-k+1,n-k)}$ is a column vector.

From the Griffiths transversality, we know that the image of the tangent map of the period map, 
$$d\P\left(\frac{\partial }{\partial z_\mu}\bigg|_z\right), \text{ for each } 1\le \mu \le N,$$
lies in the horizontal subspace
$$\bigoplus_{1\le l\le 2k-n}\mathrm{Hom}(F_z^{k-l+1}, F_z^{k-l}/F_z^{k-l+1}).$$
By the proof of Lemma \ref{lm derivative lemma}, we know that the component of 
$d\P(\frac{\partial }{\partial z_\mu}|_z)$ in the subspace $\mathrm{Hom}(F_z^{k}, F_z^{k-1}/F_z^{k})$ is identified with
$$\frac{\partial\Phi^{(n-k+1,n-k)}}{\partial z_\mu}(z).$$

The properties of the Hodge structure of Calabi--Yau type manifolds, as given in (2) of Definition \ref{defn of gcyt}, imply that $d\P(\frac{\partial }{\partial z_\mu}|_z)$ is isomorphic to $\mathrm{Hom}(F_z^{k}, F_z^{k-1}/F_z^{k})$.

Therefore, if we define
\begin{equation}\label{canonical coordinate}
z^{c}=(z^{c}_{1},\cdots, z^{c}_{N})=\big(\Phi^{(n-k+1,n-k)}(z)\big)^{T},
\end{equation}
then the Jacobian 
$$\frac{\partial z^{c}}{\partial z}$$
is non-degenerate.
This implies that $\{U;z^c=(z_1^c,\cdots,z_N^c)\}$ is also a coordinate around the base point $p$.

\begin{definition}\label{canonical coordinate definition}
We call the coordinate $\{U;z^c=(z_1^c,\cdots,z_N^c)\}$, defined by the equations \eqref{canonical coordinate}, the canonical coordinate around the base point $p$.
\end{definition}


Let us introduce the Hodge structre of Tate, which is defined by
$$T=T^{1,1}$$
such that $T^{1,1}\simeq \C$. Define $T(-1)$ to be dual of $T$, and
$$T(m)=T^{\otimes m}=T^{m,m},\, T(-m)=T(-1)^{\otimes m}=T^{-m,-m},$$
for positive integer $m$.

\begin{remark}\label{reduction}
Let $D$ be the period domain for Calabi--Yau type manifolds.
We can define the period domain
$$D\otimes T(k-n)=\{H\otimes T(k-n)\,:H\in D\}.$$
Let $H\in D$ such that
$$H=H^{k,n-k}\oplus H^{k-1,n-k+1}\oplus \cdots \oplus H^{n-k,k},$$
where $\dim_{\C}H^{k,n-k}=\dim_{\C}H^{n-k,k}=1$.
Define $$H'^{p,q}=H^{p+n-k,q+n-k}\otimes T^{k-n,k-n}.$$
Then  
\begin{eqnarray*}
H\otimes T(k-n)&=& H^{k,n-k}\oplus H^{k-1,n-k+1}\oplus \cdots \oplus H^{n-k,k}\\
&=& (H^{k,n-k}\otimes T^{k-n,k-n})\oplus (H^{k-1,n-k+1}\otimes T^{k-n,k-n})\oplus \cdots \\
&&\oplus (H^{n-k,k}\otimes T^{k-n,k-n})\\
&=& H'^{2k-n,0} \oplus H'^{2k-n-1,1}\oplus \cdots \oplus H'^{0,2k-n},
\end{eqnarray*}
where $\dim_{\C}H'^{2k-n,0}=\dim_{\C}H'^{0,2k-n}=1$.
Therefore $D\otimes T(k-n)$ is the period domain for Calabi--Yau manifolds of weight $2k-n$.

Let $N'_{+}$ be the corresponding unipotent group for the period domain $D\otimes T(k-n)$.
Then the corresponding matrices in $N'_{+}\cap (D\otimes T(k-n))$ is of the form 
\begin{align*}
\Psi'=\left[\begin{array}[c]{ccccccc}
I &0&\cdots &0  \\ 
\Psi'^{(1,0)} & I          & \cdots     &0      \\ 
\vdots & \vdots & \ddots     &\vdots        \\
\Psi'^{(2k-n,0)}     &  \Psi'^{(2k-n,1)}   &  \cdots & I \\
\end{array}\right],
\end{align*}
such that $\Psi'^{(m,l)}=\Psi^{(m+n-k,l+n-k)}$.

Since we are considering the local expansions of the sections of Hodge bundles, the results are the same if we tensor the period domain globally by $T(k-n)$.
Hence we only need to prove the following results for Calabi--Yau manifolds.
\end{remark}

\begin{theorem}\label{canonical expansion theorem}
Let $X_{p}$ be a Calabi--Yau type manifold.
Under the canonical coordinate $\{U;z^c\}$ around the base point $p$, we have the following expansion of the holomorphic section $\Omega(z^c)$ of the Hodge bundle $\mathcal{H}^{k,n-k}$,
$$\Omega(z^c)=\Omega_0+\sum_{1\le i\le N}[\theta^c_i \lrcorner\Omega_0]z_i^c+\sum_{1\le i,j\le N}[\theta^c_i \lrcorner\theta^c_j \lrcorner\Omega_0]z_i^cz_j^c+O(|z^c|^3),$$
where the high order terms $O(|z^c|^3)$ lie in 
$$\bigoplus_{l\ge 2}H^{k-l,n-k+l}(X_p).$$
Here $[\ ]$ denotes the cohomological class.
\end{theorem}
\begin{proof}
By construction in \eqref{lm section}, we define the section of Hodge bundle $\mathcal{H}^{k,n-k}$ by
\begin{eqnarray}\label{canonical expansion1}
\Omega(z^c)&:=&\Omega_{(n-k)}(z^c)  \\
&=&\sum_{\beta=n-k}^{k}\eta_{(\beta)}\cdot\Phi^{(\beta,n-k)}(z^c) \nonumber \\
&=&\eta_{(n-k)}+\sum_{\beta\ge n-k+1}\eta_{(\beta)}\cdot\Phi^{(\beta,n-k)}(z^c).\nonumber
\end{eqnarray}

According to Remark \ref{reduction}, we only need to prove the theorem for Calabi--Yau manifolds, and hence we can assume that $k=n$.
Write $\Omega_{0}=\eta_{(0)}\in H_{\mathrm{pr}}^{n,0}(X_{p})$.
Then 
\begin{eqnarray}\label{canonical expansion1'}
\Omega(z^c)&:=&\Omega_{(0)}(z^c)  \\
&=&\sum_{\beta=0}^{n}\eta_{(\beta)}\cdot\Phi^{(\beta,0)}(z^c) \nonumber \\
&=&\eta_{(0)}+\sum_{\beta\ge 1}\eta_{(\beta)}\cdot\Phi^{(\beta,0)}(z^c).\nonumber\\
&=&\Omega_0+\sum_{1\le i\le N}\eta_{f^{n-1}+i-1}z_i^c + \eta_{(2)}\cdot\Phi^{(2,0)}(z^c)\nonumber \\
&&+\sum_{\beta\ge 3}\eta_{(\beta)}\cdot\Phi^{(\beta,0)}(z^c). \nonumber
\end{eqnarray}

By the proof of Griffiths transversality in Pages 813--814 of \cite{Griffiths2}, we know that 
\begin{align}\label{term 3}
 \eta_{f^{n-1}+i-1}=\frac{\partial \Omega}{\partial z^c_i}(0) = \theta^c_i \lrcorner \Omega_0.
 \end{align} 

Applying Proposition \ref{lm expansion proposition} for the canonical coordinate, we have that 
\begin{align}\label{term 3}
\eta_{(2)}\cdot\Phi^{(2,0)}(z^c)=O(|z^c|^2) \in H^{n-2,2}(X_p),
\end{align} 
and 
\begin{align}\label{term 4}
\sum_{\beta\ge 3}\eta_{(\beta)}\cdot\Phi^{(\beta,0)}(z^c) =O(|z^c|^3) \in \bigoplus_{l\ge 3}H^{n-l,l}(X_p).
\end{align}
Therefore, we only need to check the third term in equation \eqref{canonical expansion1'}, i.e. 
\begin{align}\label{term 3'}
\eta_{(2)}\cdot\Phi^{(2,0)}(z^c)= \sum_{1\le i,j\le N}[\theta^c_i \lrcorner\theta^c_j \lrcorner\Omega_0]z_i^cz_j^c+O(|z^c|^3).
\end{align}

In fact, equations \eqref{term 3} and \eqref{term 4} imply that the second order term $\frac{\partial^2 \Omega}{\partial z^c_i\partial z^c_j}(0)$ of the expansion of $\Omega(z^c)$ lies in $H^{n-2,2}(X_p)$.

By the calculation in Page 813 of \cite{Griffiths2}, the second order term
$$\frac{\partial^2 \Omega}{\partial z^c_i\partial z^c_j}(0)=\theta_{ij}^c\lrcorner\Omega_0 + \theta^c_i \lrcorner\theta^c_j \lrcorner\Omega_0,$$
where $\theta_{ij}^c$ is the second order term of the expansion of $\phi(z^c)$,
such that
$$\phi(z^c)=\sum_i \theta^c_i z_i^c+\sum_{ij}\theta^c_{ij} z_i^cz_j^c+O(|z^c|^3)\in H^1(X_p,\Theta_{X_p}),$$
is the Beltrami differential which defines the complex structure on the polarized manifold near $p$.

Then we conclude that 
$$\theta_{ij}^c\lrcorner\Omega_0 =\frac{\partial^2 \Omega}{\partial z^c_i\partial z^c_j}(0)- \theta^c_i \lrcorner\theta^c_j \lrcorner\Omega_0 \in H^{n-2,2}(X_p).$$
But $\theta_{ij}^c\lrcorner\Omega_0\in H^{n-1,1}(X_p)$. We can see that, under the canonical coordinate, $\theta_{ij}^c\lrcorner\Omega_0=0$, and then
$$\frac{\partial^2 \Omega}{\partial z^c_i\partial z^c_j}(0)=\theta^c_i \lrcorner\theta^c_j \lrcorner\Omega_0.$$
Therefore, we have proved equation \eqref{term 3'}, which completes the proof of the theorem.
\end{proof}

\subsection{Ball type case}\label{Section Ball type}
In this section, we come back to the setup in Section \ref{basic assump}, i.e. we consider $X$ as a polarized manifold of ball type case in Definition \ref{ball type}.

As given in Definition \ref{ball type},
let $\Omega_0$ in $H_{\mathrm{pr}}^{k,n-k}(X)$ such that, if $\{\theta_{i}\}_{1\le i\le m}$ is the basis of $H^{1}(X,\Theta_{X})$, then 
\begin{equation}
\{\theta_{i}\lrcorner \Omega_0\}_{1\le i\le m} \textrm{ is linearly independent in } H_{\mathrm{pr}}^{k-1,n-k+1}(X), \tag{$*_1$} 
\end{equation}
and 
\begin{equation}
\theta_{i}\lrcorner\theta_{j}\lrcorner (H_{\mathrm{pr}}^{k,n-k}(X))=0, \text{ for any }1\le i,j \le m. \tag{$*_2$}
\end{equation}

Let $f:\,\mathcal{X}\to S$ be a family of ball type manifolds over the complex manifold $S$. 
As discussed in Section \ref{GCY}, 
we take a small enough neighborhood $U$ of $S$, to get the well-defined period map
$$\Phi:\, U \to D,$$
as well as the Hodge bundles $$\mathscr{H}^{l,n-l}:\, = \mathscr{F}^{l}/\mathscr{F}^{l+1}$$ over $U$, for $n-k\le l\le k$.

Fix $p\in S$ as the base point, we can choose any holomorphic coordinate $\{U;z\}$ around the base point $p$ such that $z(p)=0$. Then the basis $\{\theta_{i}\}_{1\le i\le m}$ of $H^{1}(X,\Theta_{X})$ can be given by the images of the holomorphic vectors
$$\frac{\partial}{\partial z_i}\bigg|_0\in \mathrm{T}_pS,\ 1\le i\le N,$$
under the Kodaira-Spencer map.
By equation ($*_1$), $\{\theta_{i}\lrcorner \Omega_0\}_{1\le i\le N}$ is linearly independent in $H^{k-1,n-k+1}(X)$. 

Now we devote ourselves to proving Theorem \ref{linear expansion}.
By Remark \ref{reduction}, we can assume that $k=n$ in the definition of polarized manifold of ball type.

We then choose the adapted basis
\begin{equation}\label{adapted basis of ball}
\eta=\{\eta_{(0)},\eta_{(1)},\cdots,\eta_{(n)}\},
\end{equation}
of the Hodge structure $$H_{\mathrm{pr}}^n(X_p, \C)=H_{\mathrm{pr}}^{n,0}(X_p)\oplus H_{\mathrm{pr}}^{n-1,1}(X_p)\oplus \cdots H_{\mathrm{pr}}^{0,n}(X_p)$$ 
where $\eta_{(l)}$ is the basis of $H_{\mathrm{pr}}^{n-l,l}(X_p)$ for $0\le l\le n$, such that $\Omega_0$ is contained in $\eta_{(0)}$, and that $\{\theta_{i}\lrcorner \Omega_0\}_{1\le i\le N}$ is contained in $\eta_{(1)}$.

Under the adapted basis $\eta$, we have the matrix representations $\Phi(q)$ of the image in $N_+$ of the period map for any $q\in U\subset\P^{-1}c$, which is of the form
\begin{equation}\label{blocks of 2,1}
\left(\begin{array}[c]{ccccc}1 &&&& \\ 
0& I&&& \\ 
\Phi_{b}^{(1,0)} & \Phi_{d}^{(1,0)} &I &&\\ 
\Phi_{c}^{(1,0)} & \Phi_{e}^{(1,0)} & 0&I &\\  
\vdots & \vdots & \vdots & \vdots & \ddots 
\end{array}\right),
\end{equation}
where the blocks 
$$\Phi^{(1,0)} =\left(\begin{array}[c]{cc} 
\Phi_{b}^{(1,0)} & \Phi_{d}^{(1,0)} \\ 
\Phi_{c}^{(1,0)} & \Phi_{e}^{(1,0)} \\  
\end{array}\right)$$
are defined with respect to the decomposition 
$$\{\Omega_0\} \cup \eta_{(0)}\setminus \{\Omega_0\}$$
of the basis of $H_{\mathrm{pr}}^{n,0}(X_p)$, and the decomposition 
$$\{\theta_{i}\lrcorner \Omega_0\}_{1\le i\le N} \cup \eta_{(1)}\setminus \{\theta_{i}\lrcorner \Omega_0\}_{1\le i\le N}$$
of the basis of $H_{\mathrm{pr}}^{n-1,1}(X_p)$.

We are interested in the first column of the matrix $\Phi(q)$. Hence we define the indices $I_{j}$ for $0\le j\le n$ and $\eta^{j}_{i}\in H^{n-j,j}(X_p)$ such that 
$$\eta_{(j)}=\{\eta^{j}_{i}:\, i\in I_{j}\},$$
and define $\Phi^{j}_{i}(z)$ such that 
$$\left(\big(\Phi^{j}_{i}(z)\big)_{i\in I_{j}}\right)^{T}$$
is the first column of the block $\Phi^{(j,0)}$.

Moreover, we assume that $$I_{1}=\{1,\dots , N\}\cup I^{c}_{1}$$ such that
$$\{\eta^{1}_{i}:\, 1\le i \le N\}=\{\theta_i \lrcorner\Omega_0:\, 1\le i \le N\},$$
$$\Phi_{b}^{(1,0)}=\left(\big(\Phi^{1}_{i}(z)\big)_{1\le i \le N}\right)^{T},$$
and 
$$\Phi_{c}^{(1,0)}=\left(\big(\Phi^{1}_{i}(z)\big)_{i\in I^{c}_{1}}\right)^{T}.$$


We can define the coordinate $\{U;z^c=(z_1^c,\cdots,z_N^c)\}$ around the base point $p$ by
\begin{align}\label{canonical coordinate'}
z_i^c(z)=\Phi^{1}_{i}(z),\, 1\le i \le N.
\end{align}
By Lemma \ref{lm derivative lemma}, $\{U;z^c=(z_1^c,\cdots,z_N^c)\}$ is indeed a coordinate around the base point $p$.



\begin{theorem}\label{linear expansion}
Let $X$ be the polarized manifold of ball type.
Under the canonical coordinate $\{U;z^c\}$ around the base point $p$, we have the following expansion of the holomorphic section $\Omega(z^c)$ of the Hodge bundle $\mathcal{H}^{k,n-k}$ for the polarized manifolds of ball type as
\begin{align}\label{linear expansion0}
\Omega(z^c)=\Omega_0+\sum_{1\le i\le N}[\theta^c_i \lrcorner\Omega_0]z_i^c.
\end{align}
\end{theorem}
\begin{proof}
Again, by Remark \ref{reduction}, we only need to prove the theorem for $k=n$.

By the proof of Theorem \ref{canonical expansion theorem}, we can construct the holomorphic section $\Omega(z^c)$ of the Hodge bundle $\mathscr{H}^{n,0}$ with the expansion:
\begin{eqnarray}\label{linear expansion1}
\Omega(z^c)&=&\Omega_0+\sum_{1\le i\le N}[\theta^c_i \lrcorner\Omega_0]z_i^c+\sum_{i\in I^c_1}\eta^{1}_{i}\cdot\P^{1}_i(z^c) + \sum_{i\in I_2}\eta^{2}_{i}\cdot\P^{2}_i(z^c) 
\nonumber\\
&& +\sum_{j\ge 3}\sum_{i\in I_j}\eta^{j}_{i}\cdot\P^{j}_i(z^c),
\end{eqnarray}
where the indices are defined above the theorem.

By Griffiths transversality, we see that 
$$\theta^c_j \lrcorner\Omega_0=\frac{\partial \Omega}{\partial z^c_j}(0)=\theta^c_j \lrcorner\Omega_0+\sum_{i\in I^c_1}\eta^{1}_{i}\cdot\frac{\partial \P^{1}_i}{\partial z^c_j}(0),$$
which implies that $\frac{\partial \P^{1}_i}{\partial z^c_j}(0)=0$.

For $z^c$ near the origin, by Griffiths transversality again, we have
$$\theta^c_j(z^c) \lrcorner\Omega(z^c)=\frac{\partial \Omega}{\partial z^c_j}(z^c)=\theta^c_j \lrcorner\Omega_0+\sum_{i\in I^c_1}\eta^{1}_{i}\cdot\frac{\partial \P^{1}_i}{\partial z^c_j}(z^{c})+(\cdots)\in F^{1}(X_{z^{c}})/F^{0}(X_{z^{c}}),$$
where $\theta^c_j(z^c)$ is the first oder term of the local expression of the Beltrami differential around the point $z^{c}$, and the term $(\cdots)\in \oplus_{\beta\ge 2}H^{n-\beta,\beta}(X_{p})$.

Note that, similar to the construction in \eqref{lm section}, the adapted basis of $$F^{1}(X_{z^{c}})/F^{0}(X_{z^{c}})$$
can be constructed in the forms as
\begin{eqnarray}\label{linear expansion3}
\Omega_{(1)}(z^c)&=&\eta_{(1)}+\sum_{\beta\ge 2}\eta_{(\beta)}\cdot \P^{(\beta,1)}(z^{c})\\
&&\equiv \eta_{(1)}, \text{ mod }\oplus_{\beta\ge 2}H^{n-\beta,\beta}(X_{p}), \nonumber 
\end{eqnarray}
which implies that  
$$\theta^c_j(z^c) \lrcorner\Omega(z^c)\equiv \theta^c_j \lrcorner\Omega_0 , \text{ mod }\oplus_{\beta\ge 2}H^{n-\beta,\beta}(X_{p}).$$
Therefore, by comparing types, we derive that
$$\frac{\partial \P^{1}_i}{\partial z^c_j}(z^c)=0, \text{ for } i\in I^{c}_{1},\ 1 \le j\le N.$$
Then we conclude that $\P^{1}_i(z^c)$, for $i\in I_1^c$, is constant in $U$ with value equal to $\P_i(0)=0$.

Now the expansion in equation \eqref{linear expansion1} becomes 
$$\Omega(z^c)=\Omega_0+\sum_{1\le i\le N}[\theta^c_i \lrcorner\Omega_0]z_i^c+ \sum_{i\in I_2}\eta^{2}_{i}\cdot\P^{2}_i(z^c) +\sum_{j\ge 3}\sum_{i\in I_j}\eta^{j}_{i}\cdot\P^{j}_i(z^c),$$
such that $\P^{j}_i(z^c) =O(|z|^{j-1})$, for $i\in I_{j}$.

By the calculation in Page 813 of \cite{Griffiths2} and equation $(*_{2})$, the second order term at the origin is contained in
$$\sum_{i\in I_2}\eta^{2}_{i}\cdot\P^{2}_i(z^c) = \sum_{1\le i,j\le N}[\theta^c_i \lrcorner\theta^c_j \lrcorner\Omega_0]z_i^cz_j^c + O(|z^c|^3)=O(|z^c|^3),$$
which implies that 
\begin{align}\label{linear expansion2}
\Omega(z^c)=\Omega_0+\sum_{1\le i\le N}[\theta^c_i \lrcorner\Omega_0]z_i^c+O(|z^c|^3)(\Omega),
\end{align}
such that the higher terms $O(|z^c|^3)(\Omega)$ of $\Omega$ are in $\oplus_{j\ge 3}H^{n-j,j}(X_{p})$.

By the calculation in Page 813 of \cite{Griffiths2} and equation $(*_{2})$ again, for $z^{c}$ near the origin, the second order term at $z^{c}$ is
$$\frac{\partial^{2} \Omega}{\partial z_{i}^{c}\partial z_{j}^{c}}(z^{c})=[\theta^c_i(z^{c}) \lrcorner\theta^c_j(z^{c}) \lrcorner\Omega(z^{c})]+[\theta^c_{ij}(z^{c}) \lrcorner\Omega(z^{c})]=[\theta^c_{ij}(z^{c}) \lrcorner\Omega(z^{c})],$$
where the $\theta^c_i(z^{c})$'s are the first terms of the local expansion of the Beltrami differential at the point $z^{c}$, and the $\theta^c_{ij}(z^{c})$'s are the second terms.

Note that 
$$[\theta^c_{ij}(z^{c}) \lrcorner\Omega(z^{c})]\in F^{1}(X_{z^{c}})/F^{0}(X_{z^{c}}),$$
which, combined the basis given by equation \eqref{linear expansion3}, implies that
$$[\theta^c_{ij}(z^{c}) \lrcorner\Omega(z^{c})]=(\eta_{(1)}+\sum_{\beta\ge 2}\eta_{(\beta)}\cdot \P^{(\beta,1)}(z^{c}))\cdot A,$$
for some matrix $A$. 

By equation \eqref{linear expansion2}, we have that 
$$\frac{\partial^{2} \Omega}{\partial z_{i}^{c}\partial z_{j}^{c}}(z^{c})=\frac{\partial^{2} }{\partial z_{i}^{c}\partial z_{j}^{c}}O(|z^c|^3)(\Omega) \in \bigoplus_{\beta\ge 3}H^{n-\beta,\beta}(X_{p}).$$

Therefore, we can conclude that $$\eta_{(1)}\cdot A \in \bigoplus_{\beta\ge 3}H^{n-\beta,\beta}(X_{p}),$$
which implies that $A=0$, and hence $[\theta^c_{ij}(z^{c}) \lrcorner\Omega(z^{c})]=0$ for any $z^{c}$ near the origin. Then 
$$\frac{\partial^{2} \Omega}{\partial z_{i}^{c}\partial z_{j}^{c}}(z^{c})=0$$
for any $z^{c}$ near the origin. This implies that the higher terms $O(|z^c|^3)(\Omega)$ in equation \eqref{linear expansion2} vanish, which gets the expansion \eqref{linear expansion0} of the theorem.
%
%
%
%
\end{proof}

\begin{corollary}\label{2expansion}
Under the conditions of Theorem \ref{linear expansion}, the holomorphic sections 
$$\theta^c_j(z^c) \lrcorner\Omega(z^c)\in H^{k-1,n-k+1}(X_{z^{c}})$$
is identified to $\theta^c_j \lrcorner\Omega_0$ in $H^{k-1,n-k+1}(X_{0})$ for $1\le j\le N$..
\end{corollary}

\begin{remark}
By the definition of the period map, the identifications in Corollary \ref{2expansion} between $\theta^c_j(z^c) \lrcorner\Omega(z^c)$ and $\theta^c_j \lrcorner\Omega_0$, for $1\le j\le N$, are given by the diffeomorphism
$$X_{z^{c}} \stackrel{\simeq}{\longrightarrow} X_{0},$$
for $z^{c}$ near the origin.
\end{remark}

\section{Period maps and refined period maps}
\subsection{Moduli spaces and period maps}

First we follow the notations in \cite{LS15}. We denote by $\M$ the moduli space of polarized manifolds of ball type, which is assumed to exist. 

We fix a lattice $\Lambda$ with a pairing $Q_{0}$, where $\Lambda$ is isomorphic to $H^n(X_{0},\mathbb{Z})/\text{Tor}$ for some $X_{0}$ in $\M$ and $Q_{0}$ is defined by the cup-product.
For a polarized manifold $(X,L)\in \M$, we define a marking $\gamma$ as an isometry of the lattices
\begin{equation}\label{marking}
\gamma :\, (\Lambda, Q_{0})\to (H^n(X,\mathbb{Z})/\text{Tor},Q).
\end{equation}

For any integer $m\geq 3$, we follow the definition of Szendr\"oi \cite{sz}
 to define an $m$-equivalent relation of two markings on $(X,L)$ by
$$\gamma\sim_{m} \gamma' \text{ if and only if } \gamma'\circ \gamma^{-1}-\text{Id}\in m \cdot\text{End}(H^n(X,\mathbb{Z})/\text{Tor}),$$
and denote $[\gamma]_{m}$ to be the set of all the $m$-equivalent classes of $\gamma$.
Then we call $[\gamma]_{m}$ a level $m$
structure on the polarized manifold $(X,L)$.


Let $\mathscr{L}_{m}$ be the moduli space of polarized manifolds with level $m$ structure, $m\ge 3$, which contains the given polarized manifold $(X,L)$.

\begin{definition}\label{T-class} A polarized manifold $(X,L)$ is said to belong to the algebraic T-class, if the irreducible component $\Z$ of $\mathscr{L}_{m}$ containing $(X,L)$ is a nonsingular algebraic variety over $\C$, on which there is an algebraic family $f_{m}:\,\U_{m}\to \Z$ for all $m\ge m_{0}$, where $m_{0}\ge 3$ is some integer.
\end{definition}
In this case, we may simply take $m_{0}=3$ without loss of generality.

Although $\U_m$ and $\Z$ are nonsingular algebraic varieties, we still consider the analytic topology on them, i.e. $\U_m=\U_m^{an}$ and $\Z=\Z^{an}$.

Let $\Phi_{\Z} :\Z \to \Gamma\backslash D$ be the period map for the family $f_m : \U_m \to \Z$, with $\rho:\, \pi_1({\Z}) \to \Gamma=\Gamma_{m}\subset \mathrm{Aut}(H_{\mathbb{Z}},Q)$ the monodromy representation.
Since the period map is locally liftable, we can lift the period map onto the universal cover $\T$ of $\Z$, to get the lifted period map $\Phi :\, \T \to D$ such that the diagram
$$\xymatrix{
\T \ar[r]^-{\Phi} \ar[d]^-{\pi_m} & D\ar[d]^-{\pi_{D}}\\
\Z \ar[r]^-{\Phi_{\Z}} & \Gamma\backslash D
}$$
is commutative.

Let $\T'$ be an irreducible component of the moduli space of marked and polarized manifolds containing $(X, L)$ in the T-class. We call $\T'$ the Torelli space.

By construction, the Torelli space $\T'$ is a covering space of $\mathcal{Z}_m$, with the natural covering map $\pi'_m:\, \mathcal{T}'\rightarrow\mathcal{Z}_m$ given by
$$[X, L, \gamma]\mapsto [X, L, [\gamma]_{m}],$$
where $[X, L, \gamma]$ and $[X, L, [\gamma]_{m}]$ denote the isomorphism class of polarized manifolds with markings and with level $m$ structure respectively.

From the definition of marking in \eqref{marking}, we have a well-defined period map $\Phi':\, \T'\to D$ for the pull-back family $\mathcal{U}'\to \T'$ via the covering map $\pi'_m:\, \mathcal{T}'\rightarrow\mathcal{Z}_m$.

Now we summarize the period maps defined as above in the following commutative diagram
$$\xymatrix{
\T \ar[dr]^-{\pi}\ar[rr]^-{\Phi} \ar[dd]^-{\pi_m} && D\ar[dd]^-{\pi_{D}}\\
&\T'\ar[ur]^-{\Phi'}\ar[dl]_-{\pi'_{m}}&\\
\Z \ar[rr]^-{\Phi_{\Z}} && \Gamma\backslash D,
}$$
where the  maps $\pi_{m}$, $\pi'_{m}$ and $\pi$ are all natural covering maps between the corresponding spaces.

We recall a lemma concerning the monodromy group $\Gamma$ on $\Z$ for $m\geq 3$ in \cite{LS15}.
\begin{lemma}\label{trivial monodromy}
Let $\gamma$ be the image of some element of $\pi_1(\mathcal{Z}_m)$ in $\Gamma$ under the monodromy representation. Suppose that $\gamma$ is finite, then $\gamma$ is trivial. Therefore for $m\geq 3$, we can assume that $\Gamma$ is torsion-free and $\Gamma\backslash D$ is smooth.
\end{lemma}

\subsection{Refined period map}\label{refined}
For convenience, we still assume in this subsection that $k=n$ in the definition of polarized manifold of ball type. 

Let $\Phi:\, \T \to D$ be the period map from the Teichm\"uller space $\T$.
Define $$\check{\T}:\,= \Phi^{-1}(N_{+}\cap D).$$
Define the projection map 
$$P^{1}:\, N_{+} \to N_{+}^{1},$$
where $N_{+}^{1}$ is the set of the first columns of the matrices $\Phi$ in $N_{+}$, which is of the form
$$(1,0,\Phi_{b}^{(1,0)}, \Phi_{c}^{(1,0)}, \cdots)^{T} $$ as given in \eqref{blocks of 2,1}. Then $P^{1}$ is just the corresponding projection map. 



The restricted period map $\P:\, \check{\T} \to N_{+}\cap D$ composed with the projection map 
$P^{1}:\, N_{+} \to N_{+}^{1},$
defines a map
$$\cP:\, \check{\T} \to N_{+}^{1}\cap D.$$
We call the map $\cP$ the restricted refined period map.

For any $q\in \check{\T}$, the image $\cP(q)$ is a column matrix in $N_{+}^{1}$, which is of the form
\begin{equation}\label{cpmatrix}
\cP(q)
=\left(\begin{array}[c]{cccc}1\\
0 \\   
\Phi_{b}^{(1,0)}(q) \\ 
\Phi_{c}^{(1,0)}(q) \\  
\Phi_{0}^{(2,0)}(q) \\ 
\vdots \end{array}\right),
\end{equation}
where 
$$\Phi_{0}^{(j,0)}=\left(\big(\Phi^{j}_{i}\big)_{i\in I_{j}}\right),\, j\ge 2$$
is the first column of the block $\Phi^{(j,0)}$.
Then the matrix $\cP(q)$ together with the adapted basis in \eqref{adapted basis of ball},
$$\eta=\{\eta_{(n)},\cdots,\eta_{(0)}\},$$
gives sections of $\Omega_{0}$ in $H_{\mathrm{pr}}^{n,0}(X_{p})$, which is introduced in Section \ref{Section Ball type}.


By Theorem \ref{linear expansion} in Section \ref{Section Ball type}, we have the following theorem, which is interpretation of the Theorem \ref{linear expansion} in the form of the restricted refined period map.

\begin{lemma}\label{linear expansion'}
For any $q\in \check{\T}$, the blocks of the matrix
$\cP(q)$ as given in \eqref{cpmatrix}
satisfy that 
$$\Phi_{c}^{(1,0)}(q)=0$$
 and $$\Phi_{0}^{(j,0)}(q)=0, \, j\ge 3.$$
\end{lemma}

From Theorem \ref{linear expansion},  the matrix
$\cP(q)$ is of the form
\begin{equation}\label{cpmatrix'}
\cP(q)=\left(\begin{array}[c]{cccc}1 \\
0\\  
\Phi_{b}^{(1,0)}(q) \\
0 \\  
0\\ 
\vdots 
\end{array}\right)\simeq \left(\begin{array}[c]{cccc}1 \\  \Phi_{1}^{1}(q) \\ \vdots \\ \Phi_{N}^{1}(q)\end{array}\right).
\end{equation}


The above discussion implies that the image of $\cP$ lies in a Euclidean subspace which we denote by $\mathbb{C}^N\simeq N^0_+\subseteq N^1_+$.
\begin{theorem}
The image of $\cP$ lies in the complex ball $\bB^{N}$ in $N_{+}^{0}$. Therefore, the restricted refined period map $\cP$ extends to the holomorphic 
\begin{equation}\label{refined periods}
\cP:\, \T \to \bB^{N} \subset N_{+}^{0}.
\end{equation}
\end{theorem}
\begin{proof}
We only need to prove the theorem for $k=n$.

By Theorem \ref{linear expansion}, for any $q\in \check{\T}$, the matrix $\cP(q)$ is of the form
\begin{eqnarray*}
\cP(q)&=&(1,0, \Phi_{b}^{(1,0)}(q), 0, \cdots 0)^{T}\\
&=&(1, 0,\{\Phi^{1}_{i}(q)\}_{1\le i\le N}, 0, \cdots 0)^{T}.
\end{eqnarray*}

Denote $$\Phi^{1}_{i}(q):\, = \cP_{i}(q)$$ for $1\le i\le N$. Then $\cP(q)$ is of the form
$$\cP(q)=(1, 0,\{\cP_{i}(q)\}_{1\le i\le N}, 0, \cdots 0)^{T}.$$

By the definition of the period domain, when we choose the adapted basis at the base point an orthonormal basis, we have that, from the Hodge-Riemann bilinear relations, 
$$-1+\sum_{1\le i\le N}|\cP_{i}(q)|^{2}<0,$$
which is equivalent that $\{\cP_{i}(q)\}_{1\le i\le N} \in \bB^{N}$, with $\bB^{N}$ the complex ball defined by
$$\bB^{N}:\,= \left\{w\in\C^{N}:\, \sum_{1\le i\le N}|w_{i}|^{2}<1\right\}.$$

Note that the complex ball $\bB^{N}$, as a homogeneous subspace of the period domain $D$, has the induced hyperbolic metric from the Hodge metric on $D$.

From Proposition \ref{codimension}, we see that $$\T\setminus \check{\T}=\Phi^{-1}(\check{D}\setminus N_{+})$$ is analytic subvariety of $\T$ with $\mathrm{codim}_{\C} \T\setminus \check{\T}\ge 1$. Hence by Riemann extension theorem, the map $\cP$ can extend to the holomorphic in \eqref{refined periods}.\end{proof}

\begin{definition}
We call the extended holomorphic map \eqref{refined periods} the refined period map.
\end{definition}

\begin{remark}
The refined period map can also be defined in an equivalent way that
$$\cP:\, \T \to \mathbb{P}(H),$$
such that $q\in \T$ is mapped to the complex line spanned by $\Omega_{q}$ in $H_{\mathrm{pr}}^{k,n-k}(X_{q})$, which satisfies $(*_{1})$ and $(*_{2})$ in Definition \ref{ball type}. Then the boundedness of $\cP$, defined in this way, also follows form Theorem \ref{linear expansion} and the definition of the period domain.
\end{remark}

\begin{corollary}\label{nondeg of refined}
The refined period map \eqref{refined periods} is nondegenerate.
\end{corollary}
\begin{proof}
The proof follows directly from the proof of Theorem \ref{linear expansion}.
\end{proof}

\section{Examples}\label{Examples}
Before introducing examples, let us review the Hodge theory for hypersurfaces.

Let $X$ be a smooth hypersurface of degree $d$ in $\mathbb{P}^{n+1}$, which means that $X$ is an algebraic subvariety of $\mathbb{P}^{n+1}$ determined by a polynomial $F(x_{0},\cdots,x_{n+1})$ with nondegenerate Jacobian.
Define the graded Jacobian quotient ring $R(F)$ by
$$R(F)=\C[x_{0},\cdots,x_{n+1}]\Big/\Big<\frac{\partial F}{\partial x_{0}},\cdots , \frac{\partial F}{\partial x_{n+1}}\Big>.$$
By Griffiths residue, the primitive cohomology of $X$ is
$$H^{k,n-k}(X)\simeq R(F)^{d(n+1-k)-n-2},$$
where $R(F)^{l}$ is the graded piece of degree $l$ of $R(F)$, $l\ge 0$.

The infinitesimal Torelli theorem holds for the smooth hypersurface $X=(F=0)\subset \mathbb{P}^{n+1}$ of degree $d$ if the product
$$R(F)^{d}\times R(F)^{d(n+1-k)-n-2}\to R(F)^{d(n+2-k)-n-2}$$
is non-degenerate in the first factor for some $k$ between $1$ and $n$. Macaulay's theorem tells that this product map is non-degenerate in each factor as long as $0 \le d(n+2-k)-n-2 \le (n+2)(d-2)$.
Therefore the only exceptions for the infinitesimal Torelli are quadric and cubic curves, and cubic surfaces.

\begin{example}[Cubic surfaces and cubic threefolds]\label{cubic}
Let $X\subset \mathbb{P}^{3}$ be a smooth cubic surface with defining equation $F(x_{0},x_{1},x_{2}, x_{3})$. Since $H^{2}(X,\C)=H^{1,1}(X)$ and the period map is trivial, we have to consider other ways to define the period map.

Allcock, Carlson, and Toledo consider the cyclic triple covering $\tilde{X}$ of $\mathbb{P}^3$ branched along $X$, with $\tilde{X}$ defined by $\tilde{X}=(F(x_{0},x_{1},x_{2}, x_{3})+x_4^3=0)\subset \mathbb{P}^{4}$. The Hodge structure of $\tilde{X}$ is
$$H^{3}(\tilde{X},\C)=H^{2,1}(\tilde{X})\oplus H^{1,2}(\tilde{X}),$$
where $h^{2,1}(\tilde{X})=h^{1,2}(\tilde{X})=5$ is computed as follows.

Without loss of generality, we can assume that the defining equation for $X$ is
$$F(x_{0},x_{1},x_{2}, x_{3})=x_{0}^{3}+x_{1}^{3}+x_{2}^{3}+x_{3}^{3}.$$
Then 
$$R=\frac{\C[x_{0},x_{1},x_{2}, x_{3}]}{<3x_{0}^{2},3x_{1}^{2},3x_{2}^{2},3x_{3}^{2}>}.$$
Hence $H^{1}(X,\Theta_{X})\simeq R^{3}\simeq \C\{x_{0}x_{1}x_{2},x_{0}x_{1}x_{3},x_{0}x_{2}x_{3},x_{1}x_{2}x_{3}\}$.

The graded Jacobian ring for $\tilde{X}$ is 
$$\tilde{R}=\frac{\C[x_{0},x_{1},x_{2}, x_{3},x_{4}]}{<3x_{0}^{2},3x_{1}^{2},3x_{2}^{2},3x_{3}^{2},3x_{4}^{2}>}.$$
Then by the isomorphism $H^{k,n-k}(\tilde{X})\simeq \tilde{R}^{d(n+1-k)-n-2}$, here $d(n+1-k)-n-2=7-3k$, we have
$$H^{2,1}(\tilde{X})\simeq \tilde{R}^{1}\simeq \C\{x_{0},x_{1},x_{2}, x_{3},x_{4}\},$$
$$H^{1,2}(\tilde{X})\simeq \tilde{R}^{4}\simeq \C\{x_{0}x_{1}x_{2}x_{3},x_{0}x_{1}x_{2}x_{4},x_{0}x_{2}x_{3}x_{4},x_{1}x_{2}x_{3}x_{4}\}.$$

Let $\Omega\simeq  x_{4}\in H^{2,1}(\tilde{X})$, hence 
$$H^{1}(X,\Theta_{X})\lrcorner \Omega\simeq R^{3}\cdot x_{4}=\C\{x_{0}x_{1}x_{2}x_{4},x_{0}x_{1}x_{3}x_{4},x_{0}x_{2}x_{3}x_{4},x_{1}x_{2}x_{3}x_{4}\}$$
is linearly independent in $H^{1,2}(\tilde{X})$.
Equation $(*_2)$ of condition (ii) is automatically satisfied.
\\

Similarly if $Y$ is a smooth cubic threefold, then we can assume that the defining equation for $Y$ is
$$F(x_{0},x_{1},x_{2}, x_{3}, x_{4})=x_{0}^{3}+x_{1}^{3}+x_{2}^{3}+x_{3}^{3}+x_{4}^{3}.$$
Then 
$$R=\frac{\C[x_{0},x_{1},x_{2}, x_{3},x_{4}]}{<3x_{0}^{2},3x_{1}^{2},3x_{2}^{2},3x_{3}^{2},3x_{4}^{2}>}.$$
Hence 
\begin{eqnarray*}
H^{1}(Y,\Theta_{Y})&\simeq &R^{3}\\
&\simeq &\C\{x_{0}x_{1}x_{2}, x_{0}x_{1}x_{3}, x_{0}x_{1}x_{4}, x_{0}x_{2}x_{3}, x_{0}x_{2}x_{4}, \\
&&x_{0}x_{3}x_{4}, x_{1}x_{2}x_{3}, x_{1}x_{2}x_{4}, x_{1}x_{3}x_{4}, x_{2}x_{3}x_{4}\}.
\end{eqnarray*}
The graded Jacobian ring for $\tilde{Y}$ is 
$$\tilde{R}=\frac{\C[x_{0},x_{1},x_{2}, x_{3},x_{4},x_{5}]}{<3x_{0}^{2},3x_{1}^{2},3x_{2}^{2},3x_{3}^{2},3x_{4}^{2},3x_{5}^{2}>}.$$
Then by the isomorphism $H^{k,n-k}(\tilde{Y})\simeq \tilde{R}^{d(n+1-k)-n-2}$, here $d(n+1-k)-n-2=9-3k$, we have
$$H^{3,1}(\tilde{X})\simeq \tilde{R}^{0}\simeq \C,$$
$$H^{2,2}(\tilde{X})\simeq \tilde{R}^{3}\simeq \C\{x_{0}x_{1}x_{2},\cdots ,x_{3}x_{4}x_{5}\},$$
$$H^{3,1}(\tilde{X})\simeq \tilde{R}^{6}\simeq \C\{x_{0}x_{1}x_{2}x_{3}x_{4}x_{5}\}$$
where $\{x_{0}x_{1}x_{2},\cdots, ,x_{3}x_{4}x_{5}\}$ is the set of all the monomials of degree $3$ in $$\C[x_{0},x_{1},x_{2}, x_{3},x_{4},x_{5}],$$ such that $x_{i}^{2}=0$. Hence the cardinal is $C_{5}^{3}=20$.

Let $\Omega\simeq  1\in \C\simeq H^{2,1}(\tilde{Y})$, hence 
$$H^{1}(Y,\Theta_{Y})\lrcorner \Omega\simeq R^{3}\cdot 1=R^{3}$$
is linearly independent in $H^{2,2}(\tilde{Y})\simeq \tilde{R}^{3}$.
Note that $R^{3}\subset \tilde{R}^{3}$ is the set of all the monomials of degree $3$ in $\C[x_{0},x_{1},x_{2}, x_{3},x_{4}]$.
Since there are $5$ variables in $R^{3}$, any two monomials $x_{i}x_{j}x_{k}$ and $x_{i'}x_{j'}x_{k'}$ of degree $3$ in $R^{3}$ must have the same variable, say $x_{i}=x_{i'}$. Then $x_{i}x_{j}x_{k} \cdot x_{i'}x_{j'}x_{k'}=x_{i}^{2}x_{j}x_{k}x_{j'}x_{k'}=0$ in $H^{3,1}(\tilde{X})\simeq \tilde{R}^{6}$.
This implies that equation $(*_2)$ of condition (ii) is satisfied.

\begin{theorem}[Cubic surfaces and cubic threefolds]
Let $\T$ be the Teichm\"uller space for cubic surfaces or cubic threefolds. Then the refined period map
$$\cP:\, \T \to \bB^{N} \subset N_{+}^{0} $$
is locally isomorphic.
\end{theorem}

\end{example}

Before introducing further examples, let us introduce the special case of the refined period map as follows, which is also called eigenperiod map.

\begin{example}[Eigenperiod map]
Let $\T$ be the Teichm\"uller space of $(X, L)$, and $g:\, \U \to \T$ be the analytic family. Let $G$ be a finite abelian group acting holomorphically on $\U$, preserving the line bundle $\mathcal{L}$ on $\U$. Recall that the line bundle $\mathcal{L}$ defines a polarization $L_{q}=\mathcal{L}|_{X_{q}}$ on $X_{q}=g^{-1}(q)$ for any $q\in \T$.
We assume that $g(X_{q})=X_{q}$ for any $q\in \T$ and $g\in G$.
Fix $p\in \T$ and $o=\Phi(p)\in D$ as the base points. Let $H_{\mathbb{Z}}=H^{n}(X_{p},\mathbb{Z})/\text{Tor}$ with the Poincar\'e paring $Q$. Then the simply-connectedness of $\T$ implies that we can identify $$(H^{n}(X_{q},\mathbb{Z})/\text{Tor},Q)\stackrel{\sim}{\longrightarrow}(H_{\mathbb{Z}},Q)$$ for any $q\in \T$.
Since $G$ preserves the line bundle $\mathcal{L}$ on $\U$, we have a induced action of $G$ on $H_{\mathbb{Z}}$ preserving $Q$, i.e. we have a representation $\rho:\, G\to \text{Aut}(H_{\mathbb{Z}},Q)$.

Let $H=H_{\mathbb{Z}}\otimes \C$.
Define $D^{\rho}$ by
$$D^{\rho}=\{ (F^n\subset\cdots\subset F^0)\in D:\, \rho(g)(F^k)=F^k, 0\le k\le n, \text{ for any }g\in G \}.$$
Let $\chi\in \text{Hom}(G,\C^{*})$ be a character of $G$.
Let
$$H_{\chi}=\{v\in H:\, \rho(g)(v)=\chi(g)v, \forall g\in G\}.$$
For any $(F^n\subset\cdots\subset F^0)\in D^{\rho}$, we define $$F^{k}_{\chi}=F^{k}\cap H_{\chi}\text{ and }H^{k,n-k}_{\chi}=F^{k}\cap \bar{F}^{n-k}\cap H_{\chi},\ 0\le k\le n.$$
Then we have the decomposition
\begin{align}\label{chi decomp}
H_{\chi}=H^{n,0}_{\chi}\oplus \cdots \oplus H^{0,n}_{\chi}.
\end{align}
Note that decomposition \eqref{chi decomp} is a Hodge decomposition if and only if $\chi$ is real.

Although decomposition \eqref{chi decomp} is not Hodge decomposition for general $\chi$, it still has the restricted polarization $Q$ on $H_{\chi}$ such that
\begin{align*}
Q( Cv,\bar v)>0\text{ for any }v \in H_{\chi}\setminus \{0\}.
\end{align*}
Hence the sub-domain $D^{\rho}_{\chi}$ defined by
$$D^{\rho}_{\chi}=\{(F_{\chi}^n\subset\cdots\subset F_{\chi}^0):\,(F^n\subset\cdots\subset F^0)\in D^{\rho}\},$$
has a well-defined metric, which is the restriction of the Hodge metric on $D$.

Since the action of $G$ on any fiber $X_{q}$ is holomorphic and preserves the polarization $L_{q}$ on $X_{q}$, the period map $\Phi:\, \T \to D$ takes values in $D^{\rho}$. Then we define the eigenperiod map by
$$\Phi_{\chi}:\, \T \to D^{\rho}_{\chi},$$
which is the composition of $\Phi$ with the projection map $D^{\rho} \to D^{\rho}_{\chi}$. 


\begin{theorem}\label{eigen ball}
If there exists some $\chi\in \text{Hom}(G,\C^{*})$, such that
$$\begin{array}{lll}
\text{(i). }F_{\chi}^j=0,\text{ for }j\ge k+1;\\
\text{(ii). }\mathrm{dim}_{\C}(F_{\chi}^k)=1;\\
\text{(iii). }H^{1}(X,\Theta_{X}) \stackrel{\simeq}{\longrightarrow} \mathrm{Hom}(F_{\chi}^k, F_{\chi}^{k-1}/F_{\chi}^k);\\
\text{(iv). }F_{\chi}^j=F_{\chi}^{k-1}, \text{ for }j\le k-1,
\end{array}
$$
Then the sub-domain $D^{\rho}_{\chi}$ is a complex ball. Consequently, the eigenperiod map $\Phi_{\chi}:\, \T \to D^{\rho}_{\chi}$ is locally isomorphic.
\end{theorem}
\end{example}

\begin{example}[Arrangements of hyperplanes]\label{arrangements}
Let $m\ge n$ be positive integers. Consider the complementary set
$$U=\mathbb{P}^n\setminus \bigcup_{i=1}^m H_i,$$
where $H_i$ is a hyperplane of $\mathbb{P}^n$ for each $1\le i\le m$ such that $H_1,\cdots ,H_m$ are in general positions. That is to say that, if $H_i$ is defined by the
linear forms
\begin{align}\label{linear form}
f_i(z_0,\cdots, z_n)=\sum_{j=0}^n a_{ij}z_j, 1\le i\le m,
\end{align}
then the matrix $(a_{ij})_{1\le i\le m,0\le j\le n}$ is of full rank.

The fundamental group $\pi_{1}(U)$ of $U$ is generated by the basis $g_{1},\cdots ,g_{m}$ with the relation that
$$g_{1}\cdots g_{m}=1.$$
One can see Section 8 of \cite{DK} for the geometric meaning of each generator $g_{i}$.

Choose a set of rational numbers $\mu=(\mu_{1},\cdots,\mu_{m})$ satisfying
\begin{align*}
& 0< \mu_{i}< 1, 1\le {i}\le m;\\
& | \mu |:\,=\sum_{i=1}^{m}\mu_{i}\in \mathbb{Z}.
\end{align*}
Define $\mathcal{L}_{\mu}$ to be the local system on $U$ by the homomorphism
$$\chi :\, \pi_{1}(U)\to \C^{*}, \, g_{i}\mapsto e^{-2\pi \i\mu_{i}}.$$
Since $|\mu|\in \mathbb{Z}$, the above homomorphism is well-defined. We are interested in the cohomology $H^{*}(U,\mathcal{L}_{\mu})$ with Hodge decomposition defined as follows.

Let $d$ be the least common denominator of $\mu_{1},\cdots,\mu_{m}$.
We define  $X$ to be the smooth projective variety of $\mathbb{P}^{m-1}$ by the equations
\begin{align}\label{variety of hyperplane}
a_{1j}z_{1}^{d}+\cdots+a_{mj}z_{m}^{d}=0, 0\le j\le n,
\end{align}
where the coefficients $a_{ij}$, $1\le i\le m,0\le j\le n$ determine the hyperplanes as \eqref{linear form}, and $[z_{1},\cdots, z_{m}]$ is the homogeneous coordinate of $\mathbb{P}^{m-1}$.

Define the finite group $G$ to be
$$G=\pi_{1}(U)/\pi_{1}(U)^{d},$$
which is isomorphic to the additive group $$(\mathbb{Z}/d)^{m}\bigg /\bigg <\sum_{i=1}^{m}e_{i}\bigg>,$$ where $e_{i}$ is the generator of the $i$-th component of $(\mathbb{Z}/d)^{m}$. The
group $G$ acts on $X$ as an automorphism induced by the well-defined action on $\mathbb{P}^{m-1}$:
$$[z_{1},\cdots, z_{m}] \mapsto [\tilde{g}_{1}z_{1},\cdots,\tilde{g}_{m}z_{m}],$$
where $\tilde{g}_{i}$ is the image of $g_{i}$ in $G$, $1\le i\le m$.
The action of $G$ on $X$ also induces an action on $H^{*}(X,\mathbb{Z})$.
By the definition of $d$ and $\chi:\, \pi_{1}(U)\to \C^{*}$, $\chi$ can also be considered as a character in $\text{Hom}(G,\C^{*})$.
Let $H^{n}_{\chi}(X,\C)$ be defined in the construction of eigenperiod map, and $$H^{n}_{\chi}(X,\C)=\bigoplus_{p+q=n} H^{p,q}_{\chi}(X)$$ be the decomposition as in \eqref{chi decomp}.

Lemma 8.1 of \cite{DK} implies that $$H^{i}(U,\mathcal{L}_{\mu})=0\text{ if }i\ne n, \text{ and } H^{n}(U,\mathcal{L}_{\mu})\simeq H^{n}_{\chi}(X,\C).$$
Then the isomorphism $H^{n}(U,\mathcal{L}_{\mu})\simeq H^{n}_{\chi}(X,\C)$ gives a Hodge structure on $H^{n}(U,\mathcal{L}_{\mu})$ by defining $$H^{n}(U,\mathcal{L}_{\mu})^{p,q}=H^{p,q}_{\chi}(X),\text{ for any }p+q=n.$$

Denote $h^{p,q}_{\chi}(X)=\dim_{\C}H^{p,q}_{\chi}(X)$. Then Lemma 8.2 of \cite{DK} implies that
$$h^{p,q}_{\chi}(X)={|\mu|-1 \choose p}{m-1-|\mu| \choose q},$$
and
$$\dim_{\C}H^{n}_{\chi}(X,\C)={m-2 \choose n}.$$

Now we consider the case that $|\mu|=\mu_{1}+\cdots +\mu_{m}=n+1$. Then by direct computations
\begin{eqnarray*}
h^{n,0}_{\chi}(X)&=&{n \choose n}{m-n-2 \choose 0}=1,\\
h^{n-1,1}_{\chi}(X)&=&{n \choose n-1}{m-n-2 \choose 1}=n(m-n-2).
\end{eqnarray*}
We know that the moduli space $\mathcal{P}_{m,n}$ of the ordered sets of $m$ hyperplanes in general linear position in $\mathbb{P}^{n}$ is a quasi-projective algebraic variety of dimension $n(m-n-2)$. Note that $\mathcal{P}_{m,n}$ can be identified to the moduli space of complete intersections defined by the equations in (\ref{variety of hyperplane}).

From \cite{DK}, we can define the eigenperiod map $$\Phi_{\chi}:\, \tilde{\mathcal{P}}_{m,n} \to D_{\chi}$$ as before on the universal covering space
$\tilde{\mathcal{P}}_{m,n}$ of $\mathcal{P}_{m,n}$. Moreover from Theorem 8.3 of \cite{DK}, we know that the eigenperiod map $\Phi_{\chi}:\,
\tilde{\mathcal{P}}_{m,n} \to D_{\chi}$ is a local isomorphism, if $|\mu|=n+1$. Hence conditions (ii) and (iii) of Theorem \ref{eigen ball} are satisfied. Note that condition (i) of Theorem \ref{eigen ball} is satisfied trivially with $k$ taken as $n$. Then $D_{\chi}$ is a complex ball whenever 
\begin{equation}\label{Hyper 1}
F_{\chi}^j=F_{\chi}^{n-1}, \text{ for }j\le n-1.
\end{equation}


\begin{theorem}
Assuming that $|\mu|=\mu_{1}+\cdots +\mu_{m}=n+1$ and \eqref{Hyper 1}, the eigenperiod map $$\Phi_{\chi}:\, \tilde{\mathcal{P}}_{m,n} \to D_{\chi}$$ is locally isomorphic.
\end{theorem}

\end{example}

\end{document}